\documentclass[oneside]{amsart}
\usepackage[utf8]{inputenc}
\usepackage{amsmath,amsfonts, amssymb, amsthm}
\usepackage[dvipsnames]{xcolor}
\usepackage[english]{babel}
\usepackage{accents}
\usepackage{csquotes}
\usepackage{cases}
\usepackage{enumerate, enumitem}
\usepackage{bbm}
\usepackage{todonotes}
\usepackage{comment}
\usepackage{lmodern}
\usepackage{tensor}
\usepackage{geometry}
\usepackage{hyperref}
\usepackage[capitalise]{cleveref}

\hypersetup{
    colorlinks,
    linkcolor={red!50!black},
    citecolor={blue!50!black},
    urlcolor={blue!80!black}
}

\usepackage[hyperref=true,backref=true,backend=biber,style=alphabetic,sorting=nyt,citestyle=alphabetic,maxbibnames=99]{biblatex} 
\title[Ranked Forcing and the Length of Generalized Borel Hierarchies]{Ranked Forcing and the Length of\\  Generalized Borel Hierarchies}
\date{\today}

\author[N.\ Chapman]{Nick Chapman}
\address[Nick Chapman]
{Institut f\"ur Diskrete Mathematik und Geometrie, Technische Universit\"at Wien, Wiedner Hauptstra{\ss}e 8-10/104, 1040 Vienna, Austria}
\email{nick.steven.chapman@gmail.com}

\DeclareMathOperator{\ord}{ord}
\DeclareMathOperator{\cf}{cf}
\DeclareMathOperator{\crank}{rk}
\DeclareMathOperator{\trank}{rk}

\DeclareMathOperator{\dom}{dom}
\DeclareMathOperator{\range}{rang}
\DeclareMathOperator{\supp}{supp}
\DeclareMathOperator{\powset}{\mathcal{P}}
\DeclareMathOperator{\leaf}{leaf}
\DeclareMathOperator{\nonleaf}{nonleaf}
\DeclareMathOperator{\incomp}{\bot}
\DeclareMathOperator{\points}{\range}

\let \succ \relax
\DeclareMathOperator{\succ}{succ}
\DeclareMathOperator{\pred}{pred}
\newcommand{\minus}{\setminus}

\theoremstyle{definition}
\newtheorem{lemma}{Lemma}[section]
\newtheorem{theorem}[lemma]{Theorem}
\newtheorem*{theorem*}{Theorem}
\newtheorem{definition}[lemma]{Definition}
\newtheorem{corollary}[lemma]{Corollary}
\newtheorem{proposition}[lemma]{Proposition}
\newtheorem{claim}[lemma]{Claim}
\newtheorem{fact}[lemma]{Fact}

\theoremstyle{remark}
\newtheorem{remark}[lemma]{Remark}
\newtheorem{assumption}{Assumption}

\newcommand{\bP}{\mathbb{P}}
\newcommand{\bQ}{\mathbb{Q}}
\newcommand{\bS}{\mathbb{S}}

\newcommand{\Ord}{\mathtt{Ord}}
\newcommand{\markdef}[1]{\textit{#1}}
\newcommand{\clopen}[1]{[#1]}
\newcommand{\bSigma}{\boldsymbol{\Sigma}}
\newcommand{\bPi}{\boldsymbol{\Pi}}
\newcommand{\forces}{\mathrel{\Vdash}}
\newcommand{\defas}{\mathrel{:=}}
\newcommand{\conc}{\mathbin{^\frown}}
\newcommand{\comp}{\mathrel{||}}

\renewcommand{\restriction}{\mathord{\upharpoonright}}
\newcommand{\setOf}[2]{\left\{#1 : #2\right\}}
\newcommand{\seq}[2]{\langle #1 : #2 \rangle}

\newcommand{\bK}{\mathbb{K}}
\newcommand{\name}[1]{\protect\undertilde{#1}}
\newcommand{\aforc}{{\mathbb{B}\mathbb{M}}}

\newcommand{\saforc}{{\mathbb{B}\mathbb{M}}'}

\newcommand{\pre}[2]{\tensor[^{#1}]{{#2}}{}}

\newcommand{\ssigma}[4]{\ifthenelse{\equal{#4}{}}{}{#4\mbox{-}}\bSigma^{#1}_{#2}\ifthenelse{\equal{#3}{}}{}{(#3)}}
\newcommand{\ppi}[4]{\ifthenelse{\equal{#4}{}}{}{#4\mbox{-}}\bPi^{#1}_{#2}\ifthenelse{\equal{#3}{}}{}{(#3)}}
\newcommand{\bor}[2]{\ifthenelse{\equal{#2}{}}{}{#2\mbox{-}}\mathbf{Bor}(#1)}

\DeclareMathSymbol{\mlq}{\mathord}{operators}{'134}
\DeclareMathSymbol{\mrq}{\mathord}{operators}{'42}
\let\oldenquote\enquote
\renewcommand{\enquote}[1]{\relax\ifmmode\mlq#1\mrq\else\oldenquote{#1}\fi}

\begin{document}

\begin{abstract}
    We extend A. Miller's framework of $\alpha$-forcing to the case of a regular uncountable cardinal $\kappa = \kappa^{<\kappa}$ and apply it to study the structure of the $\kappa$-Borel hierarchy on subspaces of the generalized Baire space ${}^\kappa \kappa$. We isolate a class of iterations of $\alpha$-forcing and show that it satisfies a certain combinatorial property of admitting a sufficiently rich family of rank functions; this fact is then used to construct several models in which nontrivial constellations for the length of the $\kappa$-Borel hierarchy on multiple subspaces of ${}^\kappa \kappa$ are realized simultaneously. Finally, we provide a higher variant of Steel's forcing with tagged trees and generalize arguments of Stern to derive the exact $\kappa$-Borel complexity of certain classes of well-founded trees.
\end{abstract}

\subjclass[2020]{Primary: 03E15; Secondary: 03E35, 54H05}
\keywords{Generalized descriptive set theory, Borel hierarchy, $\alpha$-forcing, definability}

\maketitle

\section{Introduction}

The class of Borel sets is among the premier well-studied objects in descriptive set theory. Based on how complex their structure is when built up from basic clopen sets, every Borel set $B \subseteq \pre{\omega}{\omega}$ of the Baire space is located at a specific level of the Borel hierarchy - the least ordinal $\alpha$ for which $B$ is $\ssigma{0}{\alpha}{}{}$. It is a basic fact of descriptive set theory that the Borel hierarchy on the reals does not collapse, that is, for every ordinal $\alpha < \omega_1$ there exists a Borel set $B \subseteq \pre{\omega}{\omega}$ that is $\ssigma{0}{\alpha}{}{}$ but not $\ppi{0}{\alpha'}{}{}$ for any $\alpha' < \alpha$. However, when relativizing to a particular subspace $X \subseteq \pre{\omega}{\omega}$ of the reals, this might no longer be true. For example, when we direct our attention to any countable subspace $X \subseteq \pre{\omega}{\omega}$, it immediately becomes clear that any (Borel or otherwise) subset of $X$ is $\ssigma{0}{2}{X}{}$, hence the Borel hierarchy on $X$ is very short. The \textit{length} or \textit{order} $\ord(X)$ of the Borel hierarchy on $X$, that is the least ordinal $\alpha$ such that every Borel subset of $X$ is $\ssigma{0}{\alpha}{X}{}$, can be thought of as a measure of topological complexity for a space. Just like the Baire space, the Cantor space $\pre{\omega}{2}$ also satisfies $\ord(\pre{\omega}{2}) = \omega_1$, and moreover the same is also true of any $X$ containing an embedding of $\pre{\omega}{2}$.

Apart from these trivial observations, not many other statements about the ordinal $\ord(X)$ seem to provable. In fact, it turns out that the value of this ordinal is highly independent of ZFC for uncountable $X$. A thorough study of the possible values of $\ord(X)$ for subspaces of the reals was initiated by A. Miller in his work \enquote{On the length of Borel hierarchies} \cite{miller_length_1979}. A number of consistency results about modifying the Borel hierarchy are given by Miller in \cite{miller_length_1979} and \cite{miller_descriptive_1995} to demonstrate that, apart from some minor restrictions, we have almost complete freedom over setting $\ord(X)$ to be equal to some ordinal in a c.c.c.\ forcing extension. The basic building block employed here is Miller's \textit{$\alpha$-forcing}, which allows us to either adjoin to our universe a \enquote{fresh} element of $\ppi{0}{\alpha}{X}{}$ or, conversely, add a $\ppi{0}{\alpha}{X}{}$-code for specific subset of $X$. With some work, this forcing can be shown to be relatively tame and exert only an expected amount of influence over the Borel hierarchy, hence iterations of these forcing yield the desired consistency results. 

One direction in which this study can be continued is to investigate the behaviour of alternative Borel hierarchies found within the area of generalized descriptive set theory. The central conceit of this field of study is to lift familiar notions from topology, descriptive set theory or forcing theory into the uncountable, often by replacing the role of $\omega$ with an uncountable cardinal $\kappa$. As such, the central topological spaces under study become the generalized Cantor space $\pre{\kappa}{2}$ and the generalized Baire space $\pre{\kappa}{\kappa}$, each equipped with the appropriate \textit{bounded topology}. Though recently generalized descriptive set theory of a singular $\kappa$ has been rising to prominence (see e.g.\ \cite{dimonte2025} or \cite{gen_borel_sets}), throughout this work we will stick exclusively with the more commonly established assumption of a regular uncountable $\kappa = \kappa^{<\kappa}$. The natural candidate to work with in this context is the $\kappa$-Borel hierarchy on subspaces of $\pre{\kappa}{\kappa}$, which is generated from basic clopen sets in the bounded topology by taking complements and unions of size ${\leq}\kappa$. This suggests the question of whether this hierarchy behaves similarly to the classical Borel hierarchy on subspaces of $\pre{\omega}{\omega}$.

A thorough study of the properties of the $\kappa$-Borel and related hierarchies on a vast number of topological spaces has recently been initiated in joint work of the author with Claudio Agostini, Luca Motto Ros and Beatrice Pitton \cite{gen_borel_sets}. In the same work, a partial generalization of $\alpha$-forcing is provided, which allowed us to lift one of Miller's theorems.
\begin{theorem*}
    Suppose $X \subseteq \pre{\kappa}{\kappa}$ has size $|X| > \kappa$ and $1 < \alpha < \omega$ is a finite ordinal. Then there exists a ${<}\kappa$-closed, $\kappa^+$-c.c.\ forcing extension $V[G]$ such that
    \[
        V[G] \models \ord_\kappa(X) = \alpha.
    \]
\end{theorem*}
Here we now write $\ord_\kappa$ to denote the length of the $\kappa$-Borel hierarchy. The limitation to finite ordinals in the uncountable setting is a technical limitation with the definition of $\alpha$-forcing that we now resolve in this current work.

The goal of this article is to extend Miller's work into the uncountable and systematize the study of iterations of $\alpha$-forcing. To this end, we encapsulate several of his arguments into a \enquote{black box} (cf.\ \cref{th: preserve generic pi set}) that allows for applying his technique of ranked forcing to automatically derive preservation properties for iterations that yield topological consistency results about the length of Borel hierarchies. The article is structured as follows.

In Sections 2-4 we provide preliminaries, notational conventions and introduce the core notions and concepts of the article. Miller's framework of ranked forcing is introduced, which will be the tool to derive statements about $\alpha$-forcing.

Section 5 generalizes $\alpha$-forcing to the setting of a regular uncountable $\kappa = \kappa^{<\kappa}$. At limit ordinals special care has to be taken, but overall the generalized forcing satisfies analogous properties to the classical one, including sufficient regularity properties such as strategic ${<}\kappa$-closure and a strong version of the $\kappa^+$-c.c.\ Using an ad hoc iterability theorem, we show that this strengthening of the c.c.\ can be preserved in ${<}\kappa$-supported forcing iterations, enabling the use of forcing extensions via iterated $\alpha$-forcing. Crucially, \cref{subsec: preserve pi set} builds the bridge between rank arguments surrounding $\alpha$-forcing and topological properties of spaces in an iterated forcing extension.

The introduction to Section 6 gives an overview over some previously known results about iterations of $\alpha$-forcing from \cite{miller_length_1979}, \cite{miller_descriptive_1995} and \cite{gen_borel_sets}. The rest of the section contains the main contribution of this paper, in which we isolate a class of iterations which can be endowed with a sufficiently rich family of rank functions with respect to \cref{th: preserve generic pi set}. In Section 7 we reap the spoils of our work from Section 6 and construct several models for nontrivial constellations of the length $\ord_\kappa$ of the $\kappa$-Borel hierarchy. We show that a fixed space $X \subseteq \pre{\kappa}{\kappa}$ can be made to attain almost any value of $\ord_\kappa$.
\begin{theorem*}
    Suppose $X \subseteq \pre{\kappa}{\kappa}$ has size $|X| > \kappa$ and $1 < \alpha < \kappa^+$ is a successor ordinal. Then there exists a ${<}\kappa$-closed, $\kappa^+$-c.c.\ forcing extension $V[G]$ such that
    \[
        V[G] \models \ord_\kappa(X) = \alpha.
    \]
    The same is also true for $\alpha = \kappa^+$.
\end{theorem*}
The generality of the iterations we study allows for a partial result on setting $\ord_\kappa$ to a limit ordinal as well. Several variations of this theorem are possible and discussed. For example, we also construct models in which $\ord_\kappa$ is set for several spaces $X$ simultaneously and even separate the values of $\ord_\kappa$ on these spaces by their size. This gives us the following:
\begin{theorem*}
    Let $f$ be a function assigning to each cardinal $\lambda$ with $\kappa < \lambda \leq 2^\kappa$ an ordinal $1 < f(\lambda) \leq \kappa^+$ such that
    \begin{enumerate}
        \item $f$ is (not necessarily strictly) increasing,
        \item if $\lambda$ is a successor cardinal, then $f(\lambda)$ is a successor ordinal or $f(\lambda) = \kappa^+$,
        \item if $\lambda$ is a limit cardinal, then $f(\lambda) = \sup_{\lambda' < \lambda} f(\lambda')$.
    \end{enumerate}
    Then there exists a ${<}\kappa$-closed, $\kappa^+$-c.c.\ generic extension $V[G]$ of the universe such that
    \[
        V[G] \models \forall X \subseteq \pre{\kappa}{\kappa}, X \in V, |X| > \kappa: \ord_\kappa(X) = f(|X|).
    \]
\end{theorem*}
Lastly, the section concludes with a phenomenon unique to the generalized setting - the preservation of definability for ground model objects. This phenomenon cannot occur on $\omega$, showing a divergence from the classical setting when studying the uncountable. We have already studied it in \cite{gen_borel_sets}, and apply it here by constructing a closed subset of $\pre{\kappa}{\kappa}$ of every possible order ${\leq}\kappa^+$.

The last Section 8 consist of a possible generalization of J. R. Steel's forcing with tagged trees to the uncountable setting. We show that standard arguments involving this forcing can be cast in the light of the rank forcing framework and go on to use this fact to derive sharp bounds for the Borel complexity of certain classes of well-founded trees. If we let $WF_\alpha$ be the class of well-founded tree of rank less than $\alpha$, then we get:

\begin{theorem*}
    The sets $WF_\alpha$ have the following complexity as $\kappa$-Borel subsets of $\powset(\pre{<\omega}{\kappa})$:
    \begin{itemize}
        \item For every $\alpha < \kappa^+$ the set $WF_{\kappa \cdot \alpha}$ is $\ssigma{0}{2\cdot \alpha}{}{\kappa}$ but not $\ppi{0}{2 \cdot \alpha}{}{\kappa}$.
        \item For every $\alpha < \kappa^+$ and $0 < \beta < \kappa$ the set $WF_{\kappa \cdot \alpha + \beta}$ is $\ppi{0}{2\cdot \alpha + 1}{}{\kappa}$ but not $\ssigma{0}{2 \cdot \alpha + 1}{}{\kappa}$.
    \end{itemize}
\end{theorem*}

This theorem is in line with known results from the countable setting, though this may be surprising in light of the fact that the set $WF$ of all well-founded trees behaves very differently in the generalized setting from the standpoint of complexity; we remark on this in \cref{subsec: wf trees}.

\subsection{Acknowledgments}
I am grateful to Claudio Agostini for several conversations and helping formulate the criticality constraint in the definition of $\alpha$-forcing. Furthermore, I would like to thank my advisor Martin Goldstern for his support and reading an advance version of this paper.

\section{Preliminaries and notation}

We lay the groundwork of reminding ourselves of some notions from (generalized) descriptive set theory and forcing; our notational conventions are mostly standard. A cursory viewing of this section is sufficient for the experienced reader.

\subsection{Topology}

Throughout this work, we will exclusively be considering regular uncountable cardinals $\kappa$ satisfying $\kappa = \kappa^{<\kappa}$. As usual in the setting of generalized descriptive set theory, we endow the generalized Baire space $\pre{\kappa}{\kappa}$ with the \markdef{bounded topology} generated by basic clopen subsets of the form $[\eta] = \setOf{x \in \pre{\kappa}{\kappa}}{\eta \subseteq x}$ for $\eta \in \pre{<\kappa}{\kappa}$. This work is concerned exclusively with topological spaces $X \subseteq \pre{\kappa}{\kappa}$ of the generalized Baire space together with the subspace topology.

\subsection{Trees}

Suppose $\lambda$ is a cardinal. A tree $T \subseteq \pre{<\lambda}{\lambda}$ is a set closed under initial segments. For a node $\eta \in T$ we write $\succ_T(\eta)$ for the set of immediate successors of $\eta$ in $T$. Likewise, we write $\succ^2_T(\eta)$ for $\bigcup_{\eta' \in \succ_T(\eta)} \succ_T(\eta')$, and $\succ^n_T(\eta)$ for $n < \omega$ is defined similarly. If $|\eta|$ is a successor ordinal, then we write $\pred_T(\eta)$ for the immediate predecessor node of $\eta$ in the tree. In this article we mainly consider well-founded trees, i.e.\ trees $T \subseteq \pre{<\omega}{\lambda}$ which have no infinite branch. If $T$ is such a well-founded tree, then it naturally carries a rank function $\trank_T : T \to \Ord$ defined as $\trank_T(\eta) = 0$ if $\succ_T(\eta) = \emptyset$ and $\trank_T(\eta) = \sup\setOf{\trank_T(\nu) + 1}{\nu \in \succ_T(\eta)}$ otherwise. We suppress the index $T$ in all of these functions if it is clear from context. Nodes with no successors are called \markdef{leaves}. Write $\leaf(T)$ for the set of leaves of $T$ and $\nonleaf(T) = T \minus \leaf(T)$. The closed subsets $C$ of $\pre{\kappa}{\kappa}$ can be identified with the sets of branches $[T]$ through trees $T \subseteq \pre{<\kappa}{\kappa}$.

\subsection{Forcing}
We follow the convention of forcing downwards, so $q$ is stronger than $p$ if $q \leq p$. The relation $p \comp q$ denotes compatibility, while $p \incomp q$ is incompatibility. If $\bP$ is a forcing notion, then $D \subseteq \bP$ is \markdef{dense} if for every $p \in \bP$ there is a $q \leq p$, $q \in D$. It is \markdef{open} if it is downwards closed. Finally, $D$ is \markdef{predense} if for every $p \in \bP$ there is a $q \in D$ with $q \comp p$. We denote names as $\name{\tau}$, and the canonical (check) name for a ground model object is denoted as $\check{\tau}$. The standard name for the generic filter added by a forcing notion is $\name{G}$. We say $\seq{\bP_\gamma, \name{\bQ}_\gamma}{\gamma < \gamma^*}$ is a forcing iteration with supports of size ${<}\lambda$ if for every $p \in \bP_{\gamma^*}$ all but ${<}\lambda$-many entries of $p$ are (the standard name for) the trivial condition $\mathbbm{1}$. We write $\bP_{\gamma^*}$ for the limit of such an iteration. Every $\bP_{\gamma^*}$-generic filter $G$ naturally induces, for each $\gamma < \gamma^*$, a $\bP_{\gamma}$-generic filter; we denote that filter by $G \restriction \gamma$.

\section{Borel sets}

Suppose $\lambda$ is a cardinal and $X \subseteq \pre{\kappa}{\kappa}$ is a space. Then the \markdef{$\lambda$-Borel sets} on $X$ are the smallest class of subsets of $X$ containing every (relatively) open subset $O \subseteq X$ that is closed under complements and unions of size $\lambda$; write $\bor{X}{\lambda}$ for this class. As usual, the $\lambda$-Borel sets can naturally be stratified into the \markdef{$\lambda$-Borel hierarchy} by letting $\ssigma{0}{1}{X}{\lambda}$ be the class of open subsets of $X$ and
\begin{gather*}
    \ssigma{0}{\alpha}{X}{\lambda} = \setOf{\bigcup_{i < \lambda} B_i}{\setOf{B_i}{i < \lambda} \subseteq \bigcup_{\beta < \alpha} \ppi{0}{\beta}{X}{\lambda}}\\
    \ppi{0}{\alpha}{X}{\lambda} = \setOf{X \minus B}{B \in \ssigma{0}{\alpha}{X}{\lambda}}.
\end{gather*}

In this work we are almost exclusively interested in the standard $\kappa$-Borel hierarchy on subspaces of $\pre{\kappa}{\kappa}$, and give only a very brief overview of relevant properties here. For a comprehensive study of Borel hierarchies we direct the reader to \cite{gen_borel_sets}; however, we caution the reader of the fact that in the aforementioned work, a more general approach is taken which defines the $\lambda$-Borel sets as the class closed under unions of length \textbf{less than $\lambda$}. Accordingly, it refers to the standard hierarchy as the $({<})\kappa^+$-Borel hierarchy, with the $({<})\kappa$-Borel hierarchy being a related phenomenon appearing in the case of a singular cardinal $\kappa$. In this work we adopt the more widespread convention of letting $\lambda$-Borel mean $({\leq}\lambda)$-Borel.

\subsection{Coding Borel sets}

Every $\lambda$-Borel set $B \subseteq X$ can be coded by means of a well-founded tree $T \subseteq \pre{<\omega}{\lambda}$ together with a function $f: \leaf(T) \to \pre{<\kappa}{\kappa}$ by letting $\mathcal{I}^f_T(\eta) = [f(\eta)] \cap X$ if $\eta \in \leaf(T)$ and 
\[
    \mathcal{I}^f_T(\eta) = \bigcap_{\eta' \in \succ_T(\eta)} X \minus \mathcal{I}^f_T(\eta'). 
\]
If it is clear from context, we often suppress the index $T$.

A set $B \subseteq X$ is $\ppi{0}{\alpha}{X}{\lambda}$ if and only if there exists a tree $T \subseteq \pre{<\omega}{\lambda}$ and $f : \leaf(T) \to \pre{<\kappa}{\kappa}$ such that $B = \mathcal{I}^f_T(\emptyset)$.

\begin{remark} \label{rem: universal tree}
    In fact, for every ordinal $\alpha < \lambda^+$ there exists a universal such $T$; i.e.\ a $T \subseteq \pre{<\omega}{\lambda}$ such that for each $B \in \ppi{0}{\alpha}{X}{\lambda}$ there exists an $f : \leaf(T) \to \pre{<\kappa}{\kappa}$ with $B = \mathcal{I}^f_T(\emptyset)$; see also \cref{def: ta}.
\end{remark}

\subsection{Order}

For a space $X \subseteq \pre{\kappa}{\kappa}$ and cardinal $\lambda$ (usually $\lambda = \kappa$) one might reason about the complexity of the $\lambda$-Borel hierarchy on $X$. A good measure of this complexity is the notion of the \markdef{length} or \markdef{order} of the space:

\begin{definition}
    Define 
    \[
    \ord_\lambda(X) \defas \min \setOf{\alpha}{\ssigma{0}{\alpha}{X}{\lambda} = \ppi{0}{\alpha}{X}{\lambda} = \bor{X}{\lambda}}.
    \]
\end{definition}

It is easy to see that $\ord_\lambda(X)$ can be at most $\lambda^+$. If $\ord_\lambda(X) < \lambda^+$, then we say the $\lambda$-Borel hierarchy \markdef{collapses} on $X$. Several equivalent criteria for $\ord_\lambda(X) \leq \alpha$ can be found in \cite[Proposition~3.15]{gen_borel_sets}.

A criterion for the non-collapse of the $\kappa$-Borel hierarchy can be derived from the perfect set property: using a standard argument involving the existence of universal sets at each level of the hierarchy, one can prove $\ord_\kappa\left(\pre{\kappa}{\kappa}\right) = \ord_\kappa \left( \pre{\kappa}{2} \right) = \kappa^+$ (\cite[Corollary~4.8]{gen_borel_sets}); in fact, any space $X$ into which $\pre{\kappa}{2}$ maps via a $\kappa$-Borel measurable embedding must satisfy $\ord_\kappa(X) = \kappa^+$ (\cite[Theorem~4.9]{gen_borel_sets}). The latter property can be thought of as a weak variant of the $\kappa$-perfect set property.

We summarize some easy observations about $\ord_\lambda$ in the form of a proposition:

\begin{proposition} \label{prop: easy order}
    Suppose $\lambda$ is a cardinal and $X \subseteq \pre{\kappa}{\kappa}$. Then
    \begin{itemize}
        \item $\ord_\lambda(X) = 1$ is equivalent to $X$ being discrete.
        \item If $|X| \leq \lambda$, then $\ord_\lambda(X) \leq 2$.
        \item For $\lambda = \kappa$, if $X$ contains a copy of $\pre{\kappa}{2}$ (as above), then $\ord_\kappa(X) = \kappa^+$.
    \end{itemize}
\end{proposition}

\begin{proof}
    The fact that $\ord_\lambda(X) = 1$ implies every singleton $\{x\}$ for $x \in X$ is open, hence $X$ is discrete. On the other hand, if $X$ is discrete, then every subset of $X$ is open. For the second claim, note that if $|X| \leq \lambda$, then every subset of $X$ is the ${\leq}\lambda$-union of singletons, hence $\ssigma{0}{2}{X}{\lambda}$.

    The last claim is discussed above.
\end{proof}

Though some few other eclectic topological conditions that directly impact the value of $\ord$ exist, they are few and far between. For $\kappa = \omega$, Sierpi\'nski subsets of the real line always satisfy $\ord_\omega(X) = 2$, and Luzin sets satisfy $\ord_\omega(X) = 3$ (see \cite{szpilrajn1930}). In the generalized context, it is straightforward to derive a definition for a $\kappa$-Luzin set $X$ as a space of size ${>}\kappa$ which has no $\kappa$-meager subspaces of size ${>}\kappa$; here the ideal of $\kappa$-meager sets is generated via ${\leq}\kappa$-closure from the nowhere dense sets. The fact that $\ord_\kappa(X) = 3$ then continues to hold true by the same proof, since the generalized Baire space satisfies the $\kappa$-Baire category theorem. The matter of a Sierpi\' nski set is more delicate, as there is no apparent way to generalize the notion of a measure to the uncountable context; indeed, a recent result of Agostini, Barrera and Dimonte \cite{agostini2026} has laid those ambitions to rest. Nevertheless, one may attempt to isolate an \enquote{artificial} ideal $\mathcal{I}$ that satisfies some/all of the classical measure zero ideal's desirable properties\footnote{One such \enquote{generalized null ideal} was proposed by Shelah \cite{shelahParallelNullIdeal2017}.}. Under even mild similarity assumptions to the null ideal, any such ideal $\mathcal{I}$ will satisfy the fact that every $\kappa$-Borel set $B$ can be written as the union $B = F \cup I$ of a $\ssigma{0}{2}{}{\kappa}$-set $F$ and a set $I \in \mathcal{I}$, implying $\ord_\kappa(X) = 2$ for $\kappa$-Sierpi\' nski sets derived from any such ideal $\mathcal{I}$.

Apart from the mild restrictions of \cref{prop: easy order}, the value of $\ord_\kappa(X)$ turns out to be very independent of ZFC. In this article, we continue work pioneered by A. Miller \cite{miller_length_1979} \cite{miller_descriptive_1995} and construct several models in which $\ord_\kappa(X)$ takes on arbitrary ordinal values for both a fixed space $X \subseteq \pre{\kappa}{\kappa}$ or multiple such spaces simultaneously; to this end, we will use his framework of ranked forcing and study an extension of his forcing notion of $\alpha$-forcing.

\section{Ranked forcing}

Following Miller \cite[Chapter~7]{miller_descriptive_1995}, we introduce the primary auxiliary framework of this article in the form of ranked forcing notions.

\begin{definition} \label{def: rank function}
    Let $\bP$ be a forcing notion. We say $\crank : \bP \to \Ord \cup \{\infty\}$ is a \textit{rank function} on $\bP$ if there exists a strictly increasing function $h_{\crank}$ on the ordinals with the following property:
    
    For each $p \in \bP$ and ordinal $\beta$ there exists a $q \in \bP$ with $q \comp p$ and $\crank(q) < h_{\crank}(\beta)$ such that for each $r \in \bP, \crank(r) < \beta$ we have
    \[
        r \comp q \implies r \comp p. 
    \]

    We say $\crank$ is a rank function for ordinals ${<}\beta^*$ if the above is true for all $\beta < \beta^*$. By convention, $\infty$ is strictly greater than every ordinal.
\end{definition}

Given a set $A$ of atoms, we consider formulas in the infinitary propositional language $\mathcal{L}_\infty(A)$, constructed from the set of atoms by repeatedly applying negation and forming arbitrary large disjunctions and conjunctions. We will assume the set $A$ consists of formulas in the forcing language of $\bP$, or at least that the statement \enquote{$p \forces \phi$} for $\phi \in A$ has a natural interpretation; therefore, every formula $\varphi \in \mathcal{L}_\infty(A)$ can be identified as a statement in the forcing language of $\bP$. Let $\mathcal{L}_\infty(A) /_{\bP}$ be the factorization of $\mathcal{L}_\infty(A)$ by the equivalence relation of forced equivalence, i.e.\ 
\[
\varphi \sim_\bP \psi \Leftrightarrow (\bP \forces \varphi \Leftrightarrow \psi)
\]
and denote by $[\varphi]$ the equivalence class of $\varphi$ under $\sim_\bP$. 

Formulas $[\varphi] \in \mathcal{L}_\infty(A) /_\bP$ can be stratified into the usual hierarchy by setting 
\begin{gather*}
    \Sigma_0(\mathcal{L}_\infty(A) /_\bP) = \setOf{[\neg a]}{a \in A}, \Pi_0(\mathcal{L}_\infty(A) /_\bP) = \setOf{[a]}{a \in A}, \\
    \Sigma_\alpha(\mathcal{L}_\infty(A) /_\bP) = \setOf{\left[\bigvee_{i \in I} \phi_i \right]}{\setOf{[\phi_i]}{i \in I} \subseteq \bigcup_{\alpha' < \alpha} \Pi_{\alpha'}(\mathcal{L}_\infty(A) /_\bP)} \text{ and } \\
    \Pi_{\alpha}(\mathcal{L}_\infty(A) /_\bP) = \setOf{[\neg \phi]}{\phi \in \Sigma_\alpha(\mathcal{L}_\infty(A) /_\bP)}.
\end{gather*}

If $f$ is a function mapping ordinals to ordinals, inductively define the powers of $f$ as 
\[
f^0(x) = x, f^{\beta + 1}(x) = f(f^\beta(x)) \text{ and } f^\lambda(x) = \sup_{\delta < \lambda} f^\delta(x).
\]

\begin{theorem}[\protect{\cite[Lemma~7.5]{miller_descriptive_1995}}] \label{th: rank-forcing}
    Let $\bP$ be a forcing notion, $\beta, \beta_0$ ordinals and $\crank$ be a rank function on $\bP$ for ordinals ${<}h_{\crank}^\beta(\beta_0)$. Assume that for all $p \in \bP$ and $a \in A$, if $p \forces a$, then there exists a $q \comp p, \crank(q) < \beta_0$ such that $q \forces a$.

    Then for all $p \in \bP$ and $[\varphi] \in \Sigma_{1+\beta}(\mathcal{L}_\infty(A)/_\bP)$ such that $p \forces \varphi$ there exists a $q \comp p, \crank(q) < h_{\crank}^\beta(\beta_0)$ with $q \forces \varphi$.
\end{theorem}

\begin{proof}
    By induction. For $\beta = 0$ we have $\bP \forces \varphi \Leftrightarrow \bigvee_{i \in I} a_i$, with $a_i \in A$. Therefore, if $p \forces \varphi$, we may strengthen $p$ to find a $p' \leq p$ and an $i \in I$ with $p' \forces a_i$. By assumption there is a $q \comp p'$ with $\crank(q) < \beta_0$ and $q \forces a_i$, so we are finished.

    Assume now that the statement holds for all ordinals ${<}\beta$. If $[\varphi] \in \Sigma_{1+\beta}(\mathcal{L}_\infty(A) /_\bP)$, then $\bP \forces \varphi \Leftrightarrow \bigvee_{i \in I} \neg \varphi_i$ with $[\varphi_i] \in \Sigma_{1 + \beta_i}(\mathcal{L}_\infty(A) /_\bP), \beta_i<\beta$. Find a $p' \leq p$ and $i \in I$ with $p' \forces \neg \varphi_i$. Since $\crank$ is a rank function, we can find a $q \comp p'$ such that $\crank(q) < h_{\crank}^{\beta_i + 1}(\beta_0) \leq h_{\crank}^{\beta}(\beta_0)$ and for all $r$ with $\crank(r) < h_{\crank}^{\beta_i}(\beta_0)$ we have $r \comp q \Rightarrow r \comp p$. We claim $q \forces \neg \varphi_i$. If this were not the case, we could find a $q' \leq q$ with $q' \forces \varphi_i$, and by the inductive assumption there is an $r \comp q'$ with $r \forces \varphi_i$ and $\crank(r) < h_{\crank}^{\beta_i}(\beta_0)$. But now $r \comp q$ which implies $r \comp p'$, contradiction.
\end{proof}

Typically $\beta_0 = 1$.

\section{$\alpha$-forcing} \label{sec: alpha forcing}

We introduce the primary forcing-theoretic tool used to control the $\kappa$-Borel hierarchy on a space $X \subseteq \pre{\kappa}{\kappa}$. This space will be fixed throughout the section.

We begin by inductively constructing, for each $\alpha < \kappa^+$, a well-founded tree $T_\alpha \subseteq \pre{<\omega}{\kappa}$. This tree $T_\alpha$ will, from this point onward, serve as a canonical template for the $\kappa$-Borel code of a $\ppi{0}{\alpha}{}{\kappa}$ set.

\begin{definition} \label{def: ta}
    For $\alpha < \kappa^+$ define
    \begin{itemize}
        \item $T_0 \defas \{\emptyset\}$,
        \item $T_{\alpha + 1} \defas \{\emptyset\} \cup \{i \conc \eta : i < \kappa, \eta \in T_\alpha\}$ and
        \item $T_\delta \defas \{\emptyset\} \cup \{i \conc \eta : i < \cf(\delta), \eta \in T_{c_\delta(i) + 4}\}$ when $\delta$ is a limit ordinal and $c_\delta: \cf(\delta) \to \delta$ is an arbitrary but fixed cofinal map.
    \end{itemize}
\end{definition}

\begin{remark}
    The tree $T_\alpha$ is universal for the $\alpha$-th level of the $\kappa$-Borel hierarchy in the sense of \cref{rem: universal tree}. Its form is relevant for future arguments; it is not the case that any universal tree will suffice. The integer of ${+}4$ in the limit case of the above definition is optimal for \cref{rem: criticality local} to hold.
\end{remark}

Working in the uncountable brings with itself the advent of nodes whose rank is a limit ordinal of cofinality less than $\kappa$. Such nodes require special treatment; the criticality constraint in Definition~\ref{def: alpha forcing} is in place to avoid a potential issue of overdetermination in conditions.

For notational expediency we introduce some jargon.

\begin{definition}
    An ordinal $\alpha$ is a \markdef{small limit} (ordinal) if it is a limit ordinal of cofinality $\cf(\alpha) < \kappa$.
    A node $\eta \in T_\alpha$ is a small limit node if its rank in $T_\alpha$ is a small limit ordinal.
\end{definition}

\begin{definition} \label{def: alpha forcing}
    For $A, B$ disjoint subsets of $X$ and $1 < \alpha < \kappa^+$,
    define $\aforc_\alpha(A, B, X)$ as the partially ordered set of all pairs $p = \langle f_p , R_p \rangle$ such that all of the following holds:
    \begin{enumerate}[label = {\alph*})]
        \item\label{def: alpha forcing-a} $f_p : \leaf(T_\alpha) \to \pre{<\kappa}{\kappa}$ is a partial function with $|f_p| < \kappa$.
        \item\label{def: alpha forcing-b} $R_p \subseteq \nonleaf(T_\alpha) \times X$, $|R_p| < \kappa$.
        \item\label{def: alpha forcing-c} If $\nu \in \succ(\eta)$ and $\langle \eta, x\rangle \in R_p$, then
        \begin{itemize}
            \item $\nu \in \nonleaf(T_\alpha)$ implies $\langle \nu, x\rangle \notin R_p$.
            \item $\nu \in \leaf(T_\alpha) \cap \dom(f_p)$ implies $x \notin \clopen{f_p(\nu)}$.
        \end{itemize}
        \end{enumerate}
    We also have two constraints involving the parameters $A,B$.
    \begin{enumerate}[resume*]
        \item\label{def: alpha forcing-d} $\setOf{x}{\langle \emptyset, x \rangle \in R_p} \cap B = \emptyset$.
        \item\label{def: alpha forcing-e} $\setOf{x}{\exists \eta \in \succ_{T_\alpha}(\emptyset) :  \langle \eta,x \rangle \in R_p} \cap A = \emptyset$.
    \end{enumerate}

    Before proceeding, we define an additional notion. A node $\eta \in T_\alpha$ is called $x$-critical (in $p$) if the following three conditions are satisfied.
    \begin{itemize}
        \item $\eta$ is a small limit node,
        \item $\langle \pred(\eta), x \rangle \in R_p$ or ($\pred(\eta) = \emptyset \wedge x \in A$) or ($\eta = \emptyset \wedge x \in B$) 
        \item there exists an $\eta' \in \succ^2(\eta)$ with $\langle \eta', x \rangle \in R_p$.
    \end{itemize}
    We simply say $\eta$ is critical if it is $x$-critical for some $x \in X$. We are now equipped to formulate the last constraint in the definition of a condition, called the \markdef{criticality constraint}:
    \begin{enumerate}[resume*]
        \item \label{def: alpha forcing-f} If $\eta$ is $x$-critical in $p$, then there exists a $\nu \in \succ(\eta)$ with $\langle \nu', x \rangle \notin R_p$ for each $\nu' \in \succ(\nu)$.
    \end{enumerate}
    
    The ordering is given by $q \leq p$ if and only if $f_p \subseteq f_q$ and $R_p \subseteq R_q$. The greatest element $\mathbbm{1}$ of the partial order is given by the pair $\langle \emptyset, \emptyset\rangle$.
\end{definition}

We also define a strict version of the poset.

\begin{definition}
    \markdef{Strict} $\alpha$-forcing is the partial order $\saforc_\alpha(A,B,X)$ containing all conditions $p = \langle f_p, R_p \rangle$ which satisfy the above properties, as well as the following strict criticality constraint:
    \begin{enumerate}
        \item[f')] If $\eta$ is $x$-critical in $p$, then there exists a $\nu \in \succ(\eta)$ with $\langle \nu, x \rangle \in R_p$.
    \end{enumerate}

\end{definition}

If ordered by the same relation as $\aforc_\alpha(A,B,X)$, the set $\saforc_\alpha(A,B,X)$ forms a forcing order. Since every condition $p \in \aforc_\alpha(A,B,X)$ can clearly be strengthened to a strict condition $p' \leq p$, the forcing $\saforc_\alpha(A,B,X)$ is a dense subforcing of $\aforc_\alpha(A,B,X)$, so they are near interchangeable from a forcing-theoretic perspective. Although $\aforc_\alpha(A,B,X)$ will turn out to merely be strategically ${<}\kappa$-closed, for technical reasons we are compelled to use it over the ${<}\kappa$-closed order of strict conditions.

\begin{remark}
    For $\kappa = \omega$ our definition is the same as in \cite{miller_length_1979}. There are no small limit nodes, so the notions of a condition and strict condition coincide. The same is true for uncountable $\kappa$ and $\alpha < \omega$, in which case this definition extends the one given in \cite{gen_borel_sets}.
\end{remark}

\begin{definition}
    If $G$ is $\aforc_\alpha(A,B,X)$-generic, define $f_G = \bigcup\{f_p : p \in G\}$ and $R_G = \bigcup\{R_p : p \in G\}$.
\end{definition}

By a density argument $f_G$ is a total function $f_G : \leaf(T_\alpha) \to \pre{<\kappa}{\kappa}$, thus we are able to interpret every $\eta \in T_\alpha$ as a $\kappa$-Borel subset $G_\eta = \mathcal{I}^{f_G}_{T_\alpha}(\eta)$ of $X$. A pair $\langle \eta, x\rangle \in R_p$ is meant to be understood as the promise that $x$ is going to be an element of $G_\eta$. The constraints we impose on conditions are crafted to ensure the semantic coherence of the relation $R_p$. The crucial Lemma~\ref{lem: D dense} directly yields Theorem~\ref{th: R interprets correctly}, showing the relationship between the syntactic relation $R_G$ and the semantic interpretation function $\eta \to G_\eta$. The $\ppi{0}{\alpha}{}{\kappa}$ set $G_\emptyset$ is the only relevant part of the generic filter that we focus on.

Before proceeding, there is an easy but important observation to be made.

\begin{remark} \label{rem: criticality local}
    The $x$-criticality of $\eta$ is a local property for which only the immediate neighborhood of $\eta$, namely $\{\pred(\eta)\} \cup \{\eta\} \cup \succ(\eta) \cup \succ^2(\eta)$ is relevant. When defining $T_\lambda$ for limit ordinals $\lambda$, the value $c_\lambda(i) + 4$ was chosen to ensure that distinct small limit nodes are sufficiently far apart so that their neighborhoods do not overlap nontrivially. In fact, for two small limit nodes $\eta_1 \neq \eta_2$, the eleven sets
    \[
        \{\pred(\eta_i)\}, \{\eta_i\}, \succ(\eta_i), \succ^2(\eta_i), \succ^3(\eta_i), \leaf(T_\alpha)
    \]
    for $i = 1,2$ are all pairwise disjoint, with the sole exception that $\pred(\eta_1) = \pred(\eta_2)$ is possible (any node whose successor is a small limit node has $\kappa$-many small limit nodes as successors). We will often make use of this fact implicitly in proofs to ensure that case distinctions are exhaustive, for example by tacitly assuming that a successor of a small limit node cannot simultaneously be the predecessor of another small limit node.
\end{remark}

\begin{lemma} \label{lem: D dense}
    Suppose $x \in X$ and $\eta \in \nonleaf(T_\alpha)$. Then the set
     \[
        D_{\eta, x} =
        \begin{cases}
            \! \setOf{p \in \aforc_\alpha(A,B,X)}{\langle \eta , x \rangle \in R_p \vee \exists \nu \in \succ(\eta) :
            \nu\in \dom(f_p)
            \land x \in \clopen{f_p(\nu)}}
        \ &\trank(\eta)=1,\\
        \! \setOf{p \in \aforc_\alpha(A,B,X)}{\langle \eta , x \rangle \in R_p \vee \exists \nu \in \succ(\eta) : \langle \nu , x \rangle  \in R_p}
        \ &\trank(\eta)>1
        \end{cases}
    \]
    is open dense in $\aforc_\alpha(A,B, X)$.
\end{lemma}

\begin{proof}
    It is obvious that the set $D_{\eta,x}$ is open (i.e.\ downwards closed).
    
    To prove density, let a condition $p \in \aforc_\alpha(A,B,X)$ be given. We aim to find a strengthening $p' \leq p$ with $p' \in D_{\eta,x}$. If $p \in D_{\eta, x}$ we are done, so assume $p \notin D_{\eta, x}$, which in particular entails $\langle \eta, x\rangle \notin R_p$. We also without loss of generality assume that $p \in \saforc_\alpha(A,B,X)$ is a strict condition. There are several cases to distinguish.

    \begin{itemize}
        \item If $\eta=\emptyset$ and $x\in A$, then set $p' = \langle f_p, R_p\cup \langle \emptyset, x\rangle \rangle$.
        \item If $\eta$ is a small limit node which is not $x$-critical in $p$, there are three cases:
        \begin{itemize}
            \item $\eta = \emptyset$ and $x \in B$: this means that $\langle \eta' , x \rangle \notin R_p$ for each $\eta' \in \succ^2(\eta)$, thus we may pick any $\nu \in \succ(\eta)$ and set $p' = \langle f_p, R_p \cup \langle \nu, x\rangle \rangle$.
            \item $\eta \in \succ(\emptyset)$ and $x \in A$: by the same argument, we may pick any $\nu \in \succ(\eta)$ and set $p' = \langle f_p, R_p \cup \langle \nu, x\rangle \rangle$.
            \item Otherwise, may set $p' = \langle f_p, R_p \cup \langle \eta, x\rangle \rangle$, since we also assume $p \notin D_{\eta,x}$.
        \end{itemize}

        \item If $\eta$ is a small limit node which is $x$-critical in $p$, the strict criticality constraint ensures $p \in D_{\eta,x}$.

        \item Assume $\eta$ is neither a small limit node nor the successor of a small limit node. Since $\kappa$ is regular, $|f_p|< \kappa$, and $|R_p| < \kappa$, while $\eta$ has $\kappa$-many successors, we can find a $\nu \in \succ(\eta)$ such that neither $\nu$ nor any node above $\nu$ appears in $f_p$ or $R_p$. Additionally, find an $s\in \pre{<\kappa}{\kappa}$ such that $x\in \clopen{s}$ and $\clopen{s}\cap \{y\mid  \langle \eta, y \rangle\in R_p\}=\emptyset$. But this means that 
        \[
            p' = 
            \begin{cases}
                \langle f_p \cup \langle \nu, s\rangle, R_p \rangle
                \qquad &\trank(\eta)=1,\\
                \langle f_p, R_p \cup \langle \nu, x\rangle \rangle
                \qquad &\trank(\nu)>1.
            \end{cases}
        \]
        is as desired.
        
        \item In the last case, we assume $\eta$ is the successor of a small limit node $\hat{\eta}$. We do an extra preparatory step and assume that $p \in D_{\hat{\eta},x}$, as we have already proved $D_{\hat{\eta},x}$ to be dense. We proceed just as before and find a $\nu \in \succ(\eta)$ such that neither $\nu$ nor any node above $\nu$ appears in $f_p$ or $R_p$, finally setting\footnote{In this case, we can be certain that $\trank(\nu)>1$.} $p' = \langle f_p, R_p \cup \langle \nu, x\rangle \rangle$. In this situation we have to keep in mind the possibility that $\hat{\eta}$ was not $x$-critical in $p$ yet is $x$-critical in $p'$ (since $\nu \in \succ^2(\eta)$). This is not an issue however, since we took care to ensure $p \in D_{\hat{\eta}, x}$ and thus the criticality constraint is satisfied. \qedhere
    \end{itemize}
\end{proof}

\begin{remark}
    In order for \cref{lem: D dense} to hold, it is necessary to prevent a certain forbidden constellation from appearing in a condition $p = \langle f_p, R_p \rangle$, as follows: for some small limit node $\eta \in T_\alpha$ we have $\langle \pred(\eta), x \rangle \in R_p$ and for each $\eta' \in \succ(\eta)$ there exists a $\nu \in \succ(\eta')$ with $\langle \nu, x \rangle \in R_p$. It is clear that if such a configuration appears in $p$, then there can be no strengthening of $p$ inside $D_{\eta, x}$, hence the set is not dense. In turn, \cref{th: R interprets correctly} would certainly fail as well.
    The criticality condition in $\aforc_\alpha(A,B,X)$ is crafted specifically so that this constellation is avoided in the forcing, without simultaneously compromising ${<}\kappa$-closure.
\end{remark}

We are now able to prove \cref{th: R interprets correctly} and \cref{cor: collapse A}, in the same way it was done in \cite{miller_descriptive_1995} for $\kappa = \omega$ and in \cite{gen_borel_sets} for $\alpha < \omega$. We recount the proof we gave in the latter work.

\begin{theorem}
\label{th: R interprets correctly}
    Let $G$ be $\aforc_\alpha(A,B,X)$-generic. Then for each $\eta \in \nonleaf(T_\alpha)$ and $x \in X$ we have $x \in G_\eta \Leftrightarrow \langle \eta, x \rangle \in R_G$.
\end{theorem}

\begin{proof} 
    This theorem is an immediate consequence of Lemma~\ref{lem: D dense}.
    Notice that by definition, we get that an analogous statement holds about leaves $\eta \in \leaf(T_\alpha)$, namely
    \[x \in G_\eta \Leftrightarrow x \in \clopen{f_G(\eta)}.\] 

    Let us first consider the special case $\trank(\eta) = 1$; all successors of $\eta$ are leaves in $T_\alpha$. If $x \in G_\eta$, then by the semantics of the interpretation function $x \notin G_\nu = \clopen{f_G(\nu)}$ for each $\nu \in \succ(\eta)$. Since $D_{\eta, x}$ is dense by Lemma~\ref{lem: D dense}, then $G$ must meet $D_{\eta, x}$, and we can conclude $\langle \eta, x\rangle \in R_G$.

    On the other hand, if $\langle \eta,x\rangle \in R_G$, then for each $p \in G$ and $\nu \in \succ(\eta) \cap \dom(f_p)$ we have $x \notin \clopen{f_p(\nu)}$, and therefore $x \in G_\eta = \bigcap_{\nu \in \succ(\eta)} X \minus \clopen{f_G(\nu)}$.

    Let now $\eta \in T_\alpha$ be a node of rank greater than $1$ such that the statement of the lemma holds for all successors of $\eta$. Since the set $D_{\eta,x}$ is dense by Lemma \ref{lem: D dense}, we have 
    \[
        \langle \eta, x \rangle \in R_G \Leftrightarrow \forall \nu \in \succ(\eta) : \langle \nu, x \rangle \notin R_G.
    \]
    But using the inductive hypothesis this means
    \[
        \langle \eta, x \rangle \in R_G \Leftrightarrow (\forall \nu \in \succ(\eta) : \langle \nu, x \rangle \notin R_G )\Leftrightarrow (\forall \nu \in \succ(\eta) :x \notin G_\nu) \Leftrightarrow x \in G_\eta  . \qedhere
    \]
\end{proof}

\begin{corollary} \label{cor: collapse A}
    Let $G$ be $\aforc_\alpha(A,B,X)$-generic. Then $A \subseteq G_\emptyset \subseteq X \minus B$.
\end{corollary}
\begin{proof}
    The constraint involving $A$ in \cref{def: alpha forcing} tells us that $\setOf{x}{\exists \eta \in \succ_{T_\alpha}(\emptyset) :  \langle \eta,x \rangle \in R_p} \cap A$ is empty for all $p \in G$, which by Theorem~\ref{th: R interprets correctly} immediately implies that $G_\eta \cap A = \emptyset$ for all $\eta \in \succ(\emptyset)$. But this means $V[G] \models A \subseteq G_\emptyset$. The fact that $V[G] \models G_\emptyset \cap B = \emptyset$ follows similarly.
\end{proof}

In other words, forcing with $\aforc_\alpha(A, X \minus A, X)$ adds a $\ppi{0}{\alpha}{X}{\kappa}$-code for $A$. In practice, we will be interested in only two kinds of $\alpha$-forcings: the first one being $\aforc_\alpha(A,X \minus A, X)$, and the other being the forcing $\aforc_\alpha(\emptyset, \emptyset, X)$ with no restrictions on the generic set $G_\emptyset$. Though we cannot apply \cref{cor: collapse A} in this case, our hope will be that the $\ppi{0}{\alpha}{X}{\kappa}$-set $G_\emptyset$ added by this forcing is sufficiently \enquote{generic} as a subset of $X$.

\begin{remark} \label{rem: upwards downwards}
    We think of the partial order $\aforc_\alpha(\emptyset, \emptyset, X)$ as pushing \enquote{upwards}, as a fresh Borel set at the $\alpha$-th level of the hierarchy is added, which under some conditions provides a lower bound for $\ord_\kappa(X)$. Conversely, the forcing $\aforc_\alpha(A, X \minus A, X)$ pushes \enquote{downwards}, since it can be used to add a code for $A$ at the $\alpha$-th level of the Borel hierarchy on $X$, subsequently providing an upper bound for $\ord_\kappa(X)$. To pin $\ord_\kappa(X)$ to a concrete value $\alpha$, an iteration of these forcings can be used to simultaneously push both \enquote{upwards} and \enquote{downwards}.
\end{remark}

\begin{remark}
    The partial order of $\alpha$-forcing can be compared to another widely known coding forcing, namely almost disjoint forcing by Solovay (see \cite{jensen_applications_1970}, also \cite{lucke_sigma_2012} for a generalization to $\kappa$). His forcing adds a real $r_A$ that is almost disjoint to a ground model set $A \subseteq X \subseteq \pre{\kappa}{\kappa}$, hence coding the set $A$ in a $\ppi{0}{2}{X}{\kappa}$ manner. The forcing $\aforc_2(A, X \minus X, X)$ is structurally very similar. An exact analogue of almost disjoint forcing is given if the additional clause 
    \begin{enumerate}
        \item[g)] Enumerate all nodes of rank $1$ in $T_\alpha$ as $\seq{\eta_i}{i < \kappa}$. Then we require $f_p(\eta_{i+j} \conc k) = f_p(\eta_i \conc (k + j))$ for all $i,j,k < \kappa$.
    \end{enumerate}
    is added to \cref{def: alpha forcing}. I thank Lyubomyr Zdomskyy for pointing this out to me.
\end{remark}

\begin{remark}
    One can think of $\alpha$-forcing as adding $\kappa$-many $\kappa$-Cohen reals side-by-side, where the $i$-th $\kappa$-Cohen real $c_i : \kappa \to \pre{<\kappa}{\kappa}$ is given by the values of $\nu\conc i$, where $\nu$ ranges over all nodes of rank $1$ in $T_\alpha$. Unlike a ${<}\kappa$-supported product of $\kappa$-Cohen forcings, $\alpha$-forcing allows for nontrivial constraints to be placed on all $\kappa$-many of the generic reals simultaneously via the relation $R$; the value of $\alpha$ is encoded in the tree structure of $T_\alpha$. Note that unlike $\alpha$-forcing, a $\kappa$-Cohen extension cannot change $\ord_\kappa$ for ground model spaces \cite[Lemma~6.14]{gen_borel_sets}.
\end{remark}

\subsection{Regularity and iterability}

In this section we verify cardinal preservation under $\alpha$-forcing and iterations thereof.

Below we collect some properties of forcings that are standard and widely used.

\begin{definition}
    For a regular cardinal $\lambda$, a forcing notion $\bP$ is
    \begin{itemize}
        \item \markdef{${<}\lambda$-closed}, if every decreasing sequence $\seq{p_i}{i < 
        \delta}$ in $\bP$ of length $\delta < \lambda$ has a lower bound $p_\delta \leq p_i$. 
        \item \markdef{strategically ${<}\lambda$-closed}, if Player I has a winning strategy in the closure game $\Game(\bP, {<}\lambda)$, which is played as follows:
        \begin{itemize}
            \item Player I starts the game by playing the condition $p^I_0 = \mathbbm{1}_{\bP} \in \bP$.
            \item In every round $i < \lambda$, Player I starts by playing a condition $p^I_i$ with the property that $p^I_i \leq p^{II}_j$ for every $j < i$. Player II then responds by playing a $p^{II}_i \leq p^I_i$.
            \item Player II wins a run of the game if and only if for some $i < \lambda$, there is no suitable $p^I_i$ for Player I to play. 
        \end{itemize}
        \item \markdef{$\lambda$-c.c}. if every antichain in $\bP$ has size ${<}\lambda$.
        \item \markdef{$\lambda$-linked}, if $\bP = \bigcup_{i < \lambda} \bP_i$, where for each $i < \lambda$ and $p,q \in \bP_i$ we have $p \comp q$.
    \end{itemize}
\end{definition}

\begin{fact} \label{fact: incomp equiv}
    A set $P \subseteq \aforc_\alpha(A,B,X)$ has a common lower bound (i.e.\ a condition $r$ such that $r \leq p$ for all $p\in P$) if and only if $\bigcup P = \langle \bigcup_{p \in P} f_p , \bigcup_{p \in P} R_p\rangle$ is a condition. In that case, $\bigcup P$ is the greatest lower bound of $P$.
\end{fact}

\begin{lemma} \label{lem: strict forcing is closed}
    The set $\saforc_\alpha(A,B,X)$ of strict conditions is ${<}\kappa$-closed.
\end{lemma}
\begin{proof}
    Let $\langle p_i : i < \delta \rangle$ with $\delta < \kappa$ be a decreasing sequence of strict conditions in $\saforc_\alpha(A,B,X)$. We show that $p = \bigcup_{i < \delta} p_i$ is a strict condition. Since any two conditions in the sequence are compatible, it suffices to check the strict criticality constraint. 
    To do this, let $\eta$ be $x$-critical in $p$ and $j < \delta$ be the first ordinal such that $\eta$ is $x$-critical in $p_j$. Since $p_j$ is a strict condition, there is a $\nu \in \succ(\eta)$ with $\langle \nu, x \rangle \in R_{p_j} \subseteq R_{p}$.
\end{proof}

\begin{fact} \label{fact: strat closed}
    $\aforc_\alpha(A,B,X)$ is strategically ${<}\kappa$-closed.
\end{fact}
\begin{proof}
    It is easy to see and folklore that this is true for every forcing $\bP$ which has a dense subforcing $\bQ$ that is ${<}\kappa$-closed; it is a winning strategy for Player I in $\Game(\bP, {<}\kappa)$ to simply always play a condition in $\bQ$.
\end{proof}

We also observe that $\alpha$-forcing is $\kappa$-linked; it is easy to see that $\kappa$-linkedness implies the $\kappa^+$-c.c..

\begin{lemma} \label{lem: k linked}
    $\aforc_\alpha(A,B,X)$ is $\kappa$-linked.
\end{lemma}

\begin{proof}
    The proof is identical to \cite[Lemma~7.7]{gen_borel_sets}. In the proof, $T_\alpha^0$ refers to $\leaf(T_\alpha)$ and $T_\alpha^{>0}$ to $\nonleaf(T_\alpha)$, respectively.
\end{proof}

\begin{fact}
    A ${<}\kappa$-supported iteration of (strategically) ${<}\kappa$-closed forcings is (strategically) ${<}\kappa$-closed.
\end{fact}

We also wish to make sure iterations of the forcing do not collapse cardinals. Even though the preservation of the countable chain conditions in finite support iterations is always given, the same cannot be said about the analogous statement in the generalized setting - see for instance \cite{roslanowski_explicit_2018}.

We introduce a gentle ad hoc strengthening of $\kappa$-linkedness satisfied by the partial orders at hand that is sufficient to preserve the $\kappa^+$-c.c.\ in iterations of strategically ${<}\kappa$-closed forcings. Comparable properties can be found in \cite[Lemma~V.5.14]{kunen_set_2013} or \cite[Lemma~55]{brendle_cichons_2016}. \cref{prop: heart cc} is an adaptation of such arguments to the strategically ${<}\kappa$-closed case. We leave it to the interested reader to check that under $\kappa$-linkedness, both the combination of ${<}\kappa$-closed + well-met from \cite[Lemma~V.5.14]{kunen_set_2013} and ${<}\kappa$-closed + canonical lower bounds from \cite[Lemma~55]{brendle_cichons_2016} implies $(\heartsuit)$.

\begin{definition} \label{def: heart}
    A forcing notion $\bP$ satisfies $(\heartsuit)$ if
    \[
        \bP = \bigcup_{i < \kappa} \bP_i
    \]
    is $\kappa$-linked and there exists a strategy $\sigma$ for Player I in the closure game $\Game(\bP, {<}\omega)$ with the following property: whenever $\seq{p^I_n, p^{II}_n}{n < \omega}$ and $\seq{q^I_n, q^{II}_n}{n < \omega}$ are two runs of the closure game in which Player I played according to $\sigma$ such that there exists $\seq{i_n}{n < \omega}$ with $\{p^{I}_n, q^{I}_n\} \subseteq \bP_{i_n}$ for each $n < \omega$, then there exists a condition $r \in \bP$ with $r \leq p^{II}_n, q^{II}_n$ for each $n < \omega$.
\end{definition}

\begin{proposition} \label{prop: heart cc}
    Let $\langle \bP_\gamma, \name{\bQ}_\gamma : \gamma < \gamma^* \rangle$ be a ${<}\kappa$-supported iteration consisting of strategically ${<}\kappa$-closed forcing notions satisfying $(\heartsuit)$. Then its limit $\bP_{\gamma^*}$ has the $\kappa^+$-c.c.
\end{proposition} 

\begin{proof}
    Since every $\name{\bQ}_\gamma$ is forced to be $\kappa$-linked, we may choose names $\name{\bQ}_{\gamma,i}$ such that
    \[
        \bP_\gamma \forces \bigcup_{i < \kappa} \name{\bQ}_{\gamma.i} \text{ is a $\kappa$-linked partition of } \name{\bQ}_\gamma.
    \]
    Likewise, let $\name{\sigma}^\heartsuit_\gamma$ be a name for the strategy given by $(\heartsuit)$ in $\Game(\name{\bQ}_\gamma, {<}\omega)$.

    Using strategic ${<}\kappa$-closure of the forcings, for any $p \in \bP_{\gamma^*}$ we can find a $\accentset{\ast}{p} \leq p$ and a sequence $\seq{i(p,\gamma)}{\gamma \in \supp(p)}$ such that for every $\gamma \in \supp(p)$,
    \[
        \accentset{\ast}{p} \forces p(\gamma) \in \name{\bQ}_{\gamma, i(p,\gamma)}.
    \]
    Hence to any $p = p^{II}_0$ we can assign a run $\vec{p} = \seq{p^{I}_n, p^{II}_n}{n  < \omega} $ of $\Game(\bP_{\gamma^*}, {<}\omega)$ such that
    \begin{itemize}
        \item $p^{I}_0 = \mathbbm{1}_{\bP_{\gamma^*}}$ and $p^{II}_0 = p$.
        \item For each $0 < n < \omega$, $p^{II}_n = \accentset{\ast}{\left(p^I_n\right)}$.
        \item For each $0 < n < \omega$ and $\gamma < \gamma^*$, $p^{I}_n(\gamma) = \mathbbm{1}_{\name{\bQ}_\gamma}$ if $\gamma \notin \supp(p^{II}_{n-1})$; otherwise, $p^{I}_n(\gamma)$ is to be constructed from $\name{\sigma}^\heartsuit_\gamma$ using the partial play $\seq{p^{I}_j(\gamma), p^{II}_j(\gamma)}{j < n}$ of $\Game(\name{\bQ}_\gamma, {<}\omega)$. Since by induction $p^{I}_n \restriction \gamma \leq p^{I}_j \restriction \gamma, p^{II}_j \restriction \gamma$ for $j < n$, we have $p^{I}_n \restriction \gamma \forces p^{I}_n(\gamma) \in \name{\bQ}_\gamma$.
    \end{itemize}
    Let us also denote $\bigcup_{n < \omega} \supp(p^{II}_n)$ as $\Sigma(\vec{p})$; note that $|\Sigma(\vec{p})| < \kappa$. Furthermore, we write $\vec{p}(\gamma)$ for $\seq{p^{I}_n(\gamma), p^{II}_n(\gamma)}{n  < \omega}$. Finally, let $i(\vec{p}, \gamma, n)$ for $\gamma \in \Sigma(\vec{p})$ denote the $i(p^I_n, \gamma)$ found by \[
    p^{II}_n = \accentset{\ast}{\left(p^I_n\right)}.
    \]

    We can now fix a set $A \subseteq \bP_{\gamma^*}$ of size $|A| = \kappa^+$. By the generalized $\Delta$-system lemma (cf.\ \cite[Theorem~9.19]{jechSetTheory2003}), we may assume that the sets $\Sigma(\vec{p})$ for $p \in A$ form a $\Delta$-system with root $S$, i.e.\ $\Sigma(\vec{p}) \cap \Sigma(\vec{q}) = S$ for $p \neq q$. Consider now the families $\seq{i(\vec{p}, \gamma,n)}{\gamma \in S, n < \omega}$ for $p \in A$; there are only $\kappa^{<\kappa}=\kappa$-many such families, hence we can fix $p \neq q$ from $A$ such that $i(\vec{p}, \gamma, n) = i(\vec{q}, \gamma, n)$ for all $\gamma \in S, n < \omega$.
        
    To prove that $A$ is not an antichain, by induction we will now construct a condition $r$ such that 
    \[
        r \restriction \gamma \forces \forall n < \omega : r(\gamma) \leq p^{I}_n, q^{I}_n,
    \]
    so in particular $r \leq p, q$. Assume $r \restriction \gamma$ is already constructed as desired. If $\gamma \notin \Sigma(\vec{p}) \cup \Sigma(\vec{q})$, simply set $r(\gamma) = \mathbbm{1}_{\name{\bQ}_\gamma}$. If $\gamma \in S$ notice that 
    \[
        r \restriction \gamma \forces \vec{p}(\gamma) \text{ and } \vec{q}(\gamma) \text{ are runs of $\Game(\name{\bQ}_\gamma, {<}\omega)$ played according to $\name{\sigma}^\heartsuit_\gamma$}.
    \]    
    Furthermore for each $n < \omega$, $r \restriction \gamma \leq p^{II}_n \restriction \gamma$ and $p^{II}_n \restriction \gamma \forces p^{I}_n(\gamma) \in \name{\bQ}_{\gamma, i(\vec{p}, \gamma, n)} = \name{\bQ}_{\gamma, i(\vec{q}, \gamma, n)}$ and similarly for $q$. Hence $r \restriction \gamma$ forces that we can apply $(\heartsuit)$ to find $r(\gamma)$.

    If $\gamma \in \Sigma(\vec{p}) \minus \Sigma(\vec{q})$ (or vice versa), the proof is similar but easier and we instead apply $(\heartsuit)$ to $\vec{p}(\gamma)$ and $\vec{p}(\gamma)$.
\end{proof}

\begin{lemma}
    The forcing $\aforc_\alpha(A,B,X)$ satisfies $(\heartsuit)$.
\end{lemma}

\begin{proof}
    We know the forcing is $\kappa$-linked, so we only need to describe the strategy $\sigma$: suppose $\seq{p^I_n, p^{II}_n}{j < n}$ is a partial run of $\Game(\aforc_\alpha(A,B,X), {<}\omega)$. Write $q \defas p^{II}_{n-1}$. Then we construct $r \defas p^{I}_n \leq q = p^{II}_{n-1}$ such that 
    \begin{enumerate}
        \item $r \in \saforc_\alpha(A,B,X)$ is a condition in the strict partial order
        \item Whenever $\eta \in T_\alpha$ is a small limit node and there is an $x \in X$, $\nu \in \succ^2(\eta)$ such that $\langle \nu, x \rangle \in R_q$ but $\langle \eta', x \rangle \notin R_q$ for all $\eta' \in \succ(\eta)$, then either $\langle \eta, x \rangle \in R_r$ or $\langle \pred(\eta), x \rangle \in R_r$.\footnote{The reader is left to check that even in the cases $\eta \in \{\emptyset\} \cup \succ(\emptyset)$ and $x \in A \cup B$ this can be satisfied since $r$ is a strict condition.}
    \end{enumerate}

    To see that $\sigma$ witnesses $(\heartsuit)$, let $\seq{p^I_n, p^{II}_n}{n < \omega}$ and $\seq{q^I_n, q^{II}_n}{n < \omega}$ be two runs of the closure game as in the definition of $(\heartsuit)$. Set $t \defas \langle f_t, R_t \rangle$ where
    \[
        f_t \defas \bigcup_{n < \omega} f_{p^{I}_n} \cup f_{q^I_n}, R_t \defas \bigcup_{n < \omega} R_{p^I_n} \cup R_{q^I_n}.
    \]
    We are done if we prove that $t \in \aforc(A,B,X)$ is a condition. Checking all but the criticality constraint is routine and left to the reader; we allude merely to the fact that a violation of any of these constraints would reflect down to imply $p^I_n \incomp q^I_n$ for some $n < \omega$. To verify the criticality constraint, assume $\eta$ is $x$-critical in $t$ and let $n < \omega$ be the first index such that any $\nu \in \succ^2(\eta)$ appears as $\langle \nu, x \rangle \in R_{p^I_n} \cup R_{q^I_n}$; without loss of generality, it appears in $R_{p^I_n}$. If there is an $\eta' \in \succ(\eta)$ with $\langle \eta' ,x \rangle \in R_{p^I_n}$, then this instance of the criticality constraint is satisfied in $t$. Otherwise, we ensured $\langle \eta, x \rangle \in R_{p^I_n}$ or $\langle \pred(\eta), x \rangle \in R_{p^I_n}$. If $\langle \pred(\eta), x \rangle \in R_{p^I_n}$, apply (strict) criticality in $p^I_n$. Lastly, if $\langle \eta, x \rangle \in R_{p^I_n} \subseteq R_t$, then $\eta$ was not actually $x$-critical to begin with.
\end{proof}

\begin{corollary}
    All ${<}\kappa$-supported forcing iterations of $\alpha$-forcing are strategically ${<}\kappa$-closed and have the $\kappa^+$-c.c., hence they do not collapse cardinals.
\end{corollary}

Although we work with the strategically ${<}\kappa$-closed representation of $\alpha$-forcing, by \cref{lem: strict forcing is closed} it is forcing-equivalent to a ${<}\kappa$-closed partial order. The same remains true in iterations.

\begin{fact}[Folklore]
    Suppose $\langle \bP_\gamma, \name{\bQ}_\gamma : \gamma < \gamma^* \rangle$ is a ${<}\kappa$-supported forcing iteration in which $\bP_\gamma \forces \enquote{\name{\bQ}_\gamma \text{ is strategically ${<}\kappa$-closed}}$ for each $\gamma < \gamma^*$. Suppose further that $\name{\bQ}'_\gamma$ is a $\bP_\gamma$-name for a forcing such that 
    \[
        \bP_\gamma \forces \name{\bQ}'_\gamma \text{ is a strategically ${<}\kappa$-closed dense subforcing of } \name{\bQ}_\gamma.
    \]
    Then the limit $\bP_{\gamma^*}'$ of the ${<}\kappa$-supported iteration\footnote{Note that this iteration is well-defined, since by induction $\bP'_\gamma$ is a dense subforcing of $\bP_\gamma$ and thus there is an equivalent $\bP'_\gamma$-name for $\name{\bQ}'$.} $\langle \bP'_\gamma, \name{\bQ}'_\gamma : \gamma < \gamma^* \rangle$ is a dense subforcing of $\bP_{\gamma^*}$.
\end{fact}

\begin{corollary} \label{cor: strict dense iteration}
    Every ${<}\kappa$-supported iteration of $\alpha$-forcing is forcing-equivalent to an iteration of ${<}\kappa$-closed partial orders, namely the one consisting of the respective strict versions of $\alpha$-forcing.
\end{corollary}

\subsection{Preserving the generic $\Pi$-set} \label{subsec: preserve pi set}

Our strategy for constructing models satisfying various constellations of $\ord_\kappa$ will be to employ an iteration of $\alpha$-forcing. In the iteration, an upper bound for $\ord_\kappa(X)$ can be established straightforwardly by including sufficiently many iterands of the form $\aforc_\alpha(A, X \minus A, X)$ to add $\ppi{0}{\alpha}{X}{\kappa}$ codes for subsets of $X$. Establishing lower bounds $\ord_\kappa(X) \geq \alpha$ is much more difficult; our strategy is to first force with $\aforc_\alpha(\emptyset,\emptyset)$ to add a generic $\ppi{0}{\alpha}{}{\kappa}$ subset to $X$, and then use the machinery of ranked forcing to guarantee that this set does not drop in complexity throughout the iteration.

First, let us prove that a single $\alpha$-forcing is a ranked forcing notion. This is an instance of a more general proof that will be expanded upon in \cref{sec: iterating}, with the intrinsic importance of the concrete rank function(s) introduced here being reflected in Theorem~\ref{th: preserve generic pi set}. Secondly, it also serves the purpose of gently introducing the core argument the more intricate proofs in the next section are built around, liberated from the technicalities necessary to deal with iterations. 

\begin{definition} \label{def: rank function single}
    For $1 < \alpha < \kappa^+$, $H \subseteq X$ and $p \in \aforc_\alpha(\emptyset, \emptyset, X)$ define 
    \[
        \crank_H(p) \defas \sup\{\trank_{T_\alpha}(\eta) : \exists x \in X \minus H: \langle \eta, x \rangle \in R_p\}.
    \]
\end{definition}

We will prove that $\crank_H$ is a rank function on $\aforc_\alpha(\emptyset,\emptyset,X)$ with $h_{\crank_H}(\beta) = \beta + 1$. Recall that concretely, this means the following: for every $p \in \aforc_\alpha(\emptyset, \emptyset,X)$ and ordinal $\beta$ there exists a $q \in \aforc_\alpha(\emptyset,\emptyset,X)$ with $q \comp p$ and $\crank(q) \leq \beta$ such that for each $r \in \aforc_\alpha(\emptyset, \emptyset, X)$ with $\crank(r) < \beta$ we have
\[
    r \comp q \Rightarrow r \comp p.
\]
Notice that since $\crank_H(p) \leq \alpha $, it is only relevant to verify this for ordinals $\beta \leq \alpha$.

\begin{theorem} \label{th: rank function single}
    For every $1 < \alpha < \kappa^+$ and $H \subseteq X$ the function $\crank_H$ is a rank function on $\aforc_\alpha(\emptyset, \emptyset,X)$ with $h_{\crank_H}(\beta) = \beta + 1$.
\end{theorem}

\begin{proof}
    Suppose $p \in \aforc_\alpha(\emptyset,\emptyset,X)$ and $\beta \leq \alpha$ are given; we want to define a suitable $q$ as above. First, we strengthen $p$ such that the following is satisfied: for every pair $\langle \eta, x \rangle \in R_p$ such that $\trank(\eta) = \lambda > \beta$ is a limit ordinal and $\nu \in \succ(\eta)$ with $\trank(\nu) < \beta$ there exists a $\nu' \in \succ(\nu)$ with $\langle \nu' , x \rangle \in R_{p}$. This is possible, since there are ${<}\kappa$-many such $\eta, x$ and $\nu \in \succ(\eta)$ by the construction of $T_\alpha$. Furthermore, every such $\nu$ itself has $\kappa$-many successors in $T_\alpha$, since it is not a small limit node, so choose a $\nu' \in \succ(\nu)$ such that no node above $\nu'$ appears in $R_p$. Lastly, adding $\langle \nu', x\rangle$ to $R_p$ does not turn any node $x$-critical, see also Remark~\ref{rem: criticality local}.

    Now we define $q = \langle f_q, R_q \rangle$ with $f_q \defas f_p$ and
    \[
    R_q \defas \{ \langle \eta, x \rangle \in R_p : x \in H \vee \trank_{T_\alpha}(\eta) \leq \beta \}.
    \]

    We claim $q$ is as desired. First, it is self-evident that $q \in \aforc(\emptyset,\emptyset,X)$ is a condition; furthermore, we have $q \geq p$ and $\crank(q) \leq \beta$ by definition.

    Suppose now that $r \in \aforc(\emptyset,\emptyset,X)$ with $\crank(r) < \beta$ is compatible with $q$. Towards a contradiction, let us also assume it is incompatible with $p$. This means that $p \cup r$ is not a conditions, which can happen for a number of reasons. We are going to exclude all of them as a possibility. 
    \begin{itemize}
        \item $f_p \cup f_r$ is not a function. \\ This is impossible, since $f_p = f_q$.
        \item There is an $\eta \in T_\alpha$, $\nu \in \succ(\eta)$ and an $x \in X$ such that $\langle \eta, x \rangle \in R_r$ and $x \in [f_p(\eta)]$. \\ This also cannot happen because $f_p = f_q$.
        \item There is an $\eta \in T_\alpha$, $\nu \in \succ(\eta)$ and an $x \in X$ such that $\langle \eta, x \rangle \in R_r$ and $\langle \nu,x \rangle \in R_p$. \\ If $x \in H$, then we immediately get $\langle \nu, x \rangle \in R_q$, which contradicts $q \comp r$. If $x \notin H$, then $\trank(\nu) < \trank(\eta) < \beta$, so we also have $\langle \nu, x \rangle \in R_q$.
        \item There is an $\eta \in T_\alpha$, $\nu \in \succ(\eta)$ and an $x \in X$ such that $\langle \eta, x \rangle \in R_p$ and $x \in [f_r(\eta)]$. \\ Since we can generally assume $\beta > 0$ (otherwise there is nothing to show) and this situation in particular implies $\trank(\eta) = 1$, we have $\langle \eta, x \rangle \in R_q$. This is impossible.
        \item There is an $\eta \in T_\alpha$, $\nu \in \succ(\eta)$ and an $x \in X$ such that $\langle \eta, x \rangle \in R_p$ and $\langle \nu,x \rangle \in R_r$. \\ If $x \in H$, there is nothing to do and the contradiction is immediate. Otherwise, there are two cases:
        \begin{itemize}
            \item If $\trank(\eta)$ is a successor ordinal, then $\trank(\eta) = \trank(\nu) + 1$. Hence if $x \notin H$, then $\beta > \trank(\nu)$, so $\beta \geq \trank(\eta)$ and thus $\langle \eta, x \rangle \in R_q$.
            \item If $\trank(\eta)$ is a limit ordinal, then the preparation of $p$ at the beginning of the proof comes into play. Since $\beta > \trank(\nu)$, we have ensured that there exists a $\nu' \in \succ(\nu)$ so that $\langle \nu', x\rangle$ is in $R_p$, hence also $R_q$.
        \end{itemize}
        \item The only case left is that there is an $x \in X$, an $\eta \in T_\alpha$ such that $\langle \pred(\eta), x \rangle \in R_p \cup R_r$ and for each $\nu \in \succ(\eta)$ there is a $\nu' \in \succ(\nu)$ with $\langle \nu' , x \rangle \in R_p \cup R_r$. However, if $x \in H$ or $\langle \pred(\eta), x \rangle \in R_r$ (which implies $\beta$ is large enough), then the same incompatibility would appear between $q$ and $r$. Hence this is only interesting if $x \notin H$ and $\langle \pred(\eta), x \rangle \in R_p$. Likewise, if for some $\nu \in \succ(\eta)$ there is a $\nu' \in \succ(\nu)$ with $\langle \nu' ,x \rangle \in R_p$, then by the criticality condition for $p$ there has to be a $\bar{\nu} \in \succ(\eta)$ with $\langle \bar{\nu} , x \rangle \in R_p$. But from the previous cases we already know that this implies that for all $\bar{\nu}' \in \succ(\bar{\nu})$ we have $\langle \bar{\nu}', x \rangle \notin R_p \cup R_r$, so this cannot be. Thus the only remaining case is that all the $\nu'$ we found at the beginning must satisfy $\langle \nu' ,x \rangle \in R_r$. Since $x \notin H$, the difference in rank between each such $\nu'$ and its corresponding predecessor $\nu$ is $1$ and $\trank(\eta)$ is a limit, this means $\beta > \sup \setOf{\trank(\nu)}{\nu \in \succ(\eta)} = \trank(\eta)$. Hence $\beta \geq \trank(\pred(\eta)) = \trank(\eta) + 1$ and $\langle \pred(\eta), x \rangle \in R_q$. Contradiction. \qedhere
    \end{itemize}
\end{proof}

The next theorem builds the bridge between the abstract theory of ranked forcing and the structure of the Borel hierarchy. It encapsulates the core idea of several similar proofs in \cite{miller_length_1979}. Based on the role that $y^*$ performs in the proof, the core argument is also dubbed \enquote{the old switcheroo} by Miller. Note that we formulate it to show that the set $G_\emptyset(0)$ added by the first generic not only does not admit a simple definition at a lower level of the $\kappa$-Borel hierarchy, it also does not admit a simple definition at a lower level of the $\lambda$-Borel hierarchy for any $\lambda \geq \kappa$. This is relevant for future work.

\begin{theorem} \label{th: preserve generic pi set}
    Suppose $\lambda \geq \kappa$ is a cardinal and $\alpha$ is an ordinal with $1 < \alpha < \kappa^+$. Moreover, let $X \subseteq \pre{\kappa}{2}$ and $\name{\mathcal{P}}$ be a $\aforc_\alpha(\emptyset, \emptyset, X)$-name for a $\lambda^+$-c.c.\ forcing notion. Write
    \[
        \bP = \aforc_\alpha(\emptyset, \emptyset, X) \star \name{\mathcal{P}}.
    \]
    Suppose further that $\seq{\crank_i}{i \in I}$ is a family of functions such that the following is true:
    \begin{enumerate}
        \item For each $i \in I$, $\crank_i$ is a rank function on $\bP$ for ordinals ${<}\alpha$ with $h_{\crank_i}(\beta) = \beta + 1$. \label{enum: switcheroo 1}
        \item For each set $S \subseteq \bP$ of size $|S| \leq \lambda$ there exists an $i \in I$ with $\crank_i(p) = 0$ for all $p \in S$. \label{enum: switcheroo 2}
        \item For each $i \in I$ there exists an $H \subseteq X$ of size $|H| \leq \lambda$ such that $\crank_i(p) \geq \crank_H(p(0))$ for all $p \in \bP$. Here $\crank_H$ is the rank function on $\aforc_\alpha(\emptyset, \emptyset, X)$ from Definition~\ref{def: rank function single}. \label{enum: switcheroo 3}
    \end{enumerate}

    Then for each $Y \subseteq X$ with $|Y| > \lambda$,
    \[
        \bP \forces G_\emptyset(0) \cap Y \notin \ssigma{0}{\alpha}{Y}{\lambda},
    \]
    where $G_\emptyset(0)$ is the generic $\ppi{0}{\alpha}{X}{\kappa}$ subset of $X$ added by the first $\aforc_\alpha(\emptyset, \emptyset, X)$-generic filter.
\end{theorem}

\begin{proof}
    Suppose this is not the case. Let $p \in \bP$ be a condition and $\name{J}$ be a $\bP$-name for a $\ssigma{0}{\alpha}{Y}{\lambda}$ set such that $p \forces G_\emptyset(0) \cap Y = \name{J}$. Furthermore, let $T$ be a universal well-founded tree $T_\alpha \subseteq \pre{<\omega}{\lambda}$ as in \cref{rem: universal tree} for the $\lambda$-Borel hierarchy on $\pre{\kappa}{\kappa}$, i.e.\ for every $\ssigma{0}{\alpha}{}{\lambda}$ set $B \subseteq \pre{\kappa}{\kappa}$ there exists an $f : \leaf(T) \to \pre{<\kappa}{\kappa}$ such that $B = \mathcal{I}^f_T(\emptyset)$. Hence we can find a name $\name{f}$ for an assignment $\name{f} : \leaf(T) \to \pre{<\kappa}{\kappa}$ such that 
    \[
        \bP \forces \name{J} =  Y \minus \mathcal{I}^{\name{f}}_{T}(\emptyset).
    \]
    By the $\lambda^+$-c.c., there is a set $S \subseteq \bP$ of size ${\leq}\lambda$ such that for each $q \in \bP$ and $j < \lambda$ there is an $s \in S$ compatible with $q$ which simultaneously decides $\name{\tau}\restriction j$ and the statement $\enquote{G_\emptyset(0) = \name{J}}$. 
    By \eqref{enum: switcheroo 2}, there is an $i^* \in I$ such that $\crank_{i^*}(s) = 0$ for all $s \in S$.
    Finally, by \eqref{enum: switcheroo 3}, we can find an $H^* \subseteq X$ of size $|H^*| \leq \lambda$ corresponding to $i^*$. Since $|Y| > \lambda$, we will from now on fix a $y^* \in Y \minus H^*$ and write
    \[
        \bar{A} \defas \setOf{\enquote{y^* \in [\name{f}(\eta)]}}{\eta \in \leaf(T)} \cup \{\enquote{G_\emptyset(0) \neq \name{J}}\}.
    \]
    
\begin{claim}
    For every $\eta \in T$ there exists a sentence $\varphi_\eta \in \Sigma_{\crank_T(\eta)}(\mathcal{L}_\infty(\bar{A})/_\bP)$ such that
    \[
        \bP \forces \varphi_\eta \Leftrightarrow y^* \notin \mathcal{I}^{\name{f}}_T(\eta)
    \]
\end{claim}

\begin{proof}[Proof of the claim]
    By induction on the structure of $T$. For every $\eta \in \leaf(T)$ the sentence $\enquote{y^* \notin [\name{f}(\eta)]}$ is in $\Sigma_0(\mathcal{L}_\infty(\bar{A})/_\bP)$. Secondly, if we assume the statement is true for all $\nu \in \succ(\eta)$, then set $\varphi_\eta \defas \bigvee_{\nu \in \succ(\eta)} \neg \varphi_\nu$. We know $\varphi_\eta \in \Sigma_{\beta}(\mathcal{L}_\infty(\bar{A})/_\bP)$ with $\beta = \sup \setOf{\trank_T(\nu) + 1}{\nu \in \succ(\eta)} = \trank_T(\eta)$. Moreover,
    \[
        \bP \forces \varphi_\eta \Leftrightarrow (\exists \nu \in \succ(\eta) : \neg \varphi_\nu) \Leftrightarrow \left(\exists \nu \in \succ(\eta) : y^* \in \mathcal{I}^{\name{f}}_T(\nu) \right) \Leftrightarrow y^* \notin \mathcal{I}^{\name{f}}_T(\eta). \qedhere
    \]
\end{proof}

    Now we consider the condition $p' \leq p$, where $p'(0) = \langle f_{p(0)}, R_{p(0)} \cup \langle \emptyset, y^* \rangle \rangle$ and $p'(1) = p(1)$. We have
    \[
        p' \forces y^* \in G_\emptyset(0) \wedge G_\emptyset(0) = \name{J},
    \]
    and also
    \[
        \bP \forces y^* \in \name{J} \Leftrightarrow y^* \notin \mathcal{I}^{\name{f}}_T(\emptyset) \Leftrightarrow \varphi_\emptyset.
    \]
    Since $\alpha \geq 2$, it is easy to check that $\enquote{\varphi_\emptyset \wedge G_\emptyset(0) = \name{J}} \in \Sigma_{\alpha}(\mathcal{L}_\infty(\bar{A})/_{\bP})$. This means that we can apply Theorem~\ref{th: rank-forcing} to $\crank_{i^*}$, $\beta_0 = 1$ and $\varphi_\emptyset \wedge G_\emptyset(0) = \name{J}$ to learn that there is a $q \in \bP$ with $\crank_{i^*}(q) < \alpha$, $q \comp p'$ and 
    \[
        q \forces y^* \in G_\emptyset(0) \wedge G_\emptyset(0) = \name{J}.
    \]
    Since $y^* \in Y \minus H^*$ and $H^*, i^*$ satisfy \eqref{enum: switcheroo 3}, we know $\langle \emptyset, y^* \rangle \notin R_{q(0)}$. We claim that there is a $\nu \in \succ_{T_\alpha}(\emptyset)$ such that, purely from a syntactic standpoint, $q'_\nu \leq q$ defined as $q'_\nu(0) = \langle f_{q(0)}, R_{q(0)} \cup \langle \nu, y^* \rangle \rangle$ and $q'_\nu(1) = q(1)$ is a condition in $\bP$. If $\alpha$ is not a small limit ordinal, $\emptyset$ has $\kappa$-many successors in $T_\alpha$, hence there is a $\nu \in \succ_{T_\alpha}(\emptyset)$ such that no node above $\nu$ appears in $R_{q(0)}$; any such $\nu$ suffices. If $\alpha$ is a small limit ordinal, then there exists a $\nu \in \succ_{T_\alpha}(\emptyset)$ such that for every $\nu' \in \succ(\nu)$ we have $\langle \nu' , y^* \rangle \notin R_{q(0)}$. If this were not the case, then we would have $\crank_H(q(0)) = \alpha$, since for such $\nu' \in \succ(\nu)$ we have $\trank_{T_\alpha}(\nu) = \trank_{T_\alpha}(\nu') + 1$ and $\alpha = \trank_{T_\alpha}(\emptyset) = \sup \setOf{\trank_{T_\alpha}(\nu)}{\nu \in \succ(\emptyset)}$. Hence we can choose such a $\nu$ as above, and since $\trank_{T_\alpha}(\nu)$ is of the form $\gamma + 4$ for some ordinal $\gamma$, the criticality constraint is also satisfied in $q'_\nu(0)$, thus $q'_\nu$ is a condition. But this is absurd, because now
    \[
        q'_\nu \forces G_\emptyset(0) = \name{J} \wedge y^* \in \name{J} \wedge y^* \notin G_\emptyset(0). \qedhere
    \]
\end{proof}

\section{Iterating $\alpha$-forcing} \label{sec: iterating}

As will be shown in Section~\ref{sec: applications}, several consistency results and surgical modifications surrounding the Borel hierarchy on subspaces of $\pre{\kappa}{\kappa}$ can be achieved by suitably applying $\alpha$-forcing. Throughout these applications, the core argument yielding lower bounds for the Borel complexity of certain sets will boil down to an observation that a specific iteration of $\alpha$-forcing, or at least a dense subforcing thereof consisting of well-behaved conditions, is a ranked forcing notion in the sense of Definition~\ref{def: rank function} for a class of rank functions admitting a parameter (see e.g.\ Definition~\ref{def: locus ground}) so that \cref{th: preserve generic pi set} applies. The parameter serves as a way to exclude a \enquote{small} (generally: of size ${\leq}\kappa$) set of conditions from being ranked; this has to be done in a way that the structure of the iteration is not disturbed too much, so that enough relevant information about subsequent iterands is preserved.

Our starting point for all that follows herein will be ${<}\kappa$-supported iteration
\[
\langle \bP_\gamma, \name{\bQ}_\gamma : \gamma < \gamma^* \rangle
\]
of $\alpha$-forcing with $\bP_\gamma \forces \name{\bQ}_\gamma = \aforc_{\alpha_\gamma}(\name{A}_\gamma,\name{B}_\gamma,\name{X}_\gamma)$. The first forcing is of the form $\bQ_0 = \aforc_{\alpha_0}(\emptyset, \emptyset, X_0)$. We write $\bP = \bP_{\gamma^*}$ for the ${<}\kappa$-limit of the iteration. This iteration and its parameters will be fixed for the rest of the section.

\begin{remark} \label{rem: no general ranks}
It soon becomes clear that finding rank functions for an iteration at this level of generality is fruitless. This is already true for a two-step iteration; suppose one is given the forcing $\bP = \aforc_{\alpha}(\emptyset,\emptyset,X) \star \aforc_{\alpha}(X \minus G_\emptyset(0), G_\emptyset(0), X)$, where $G_\emptyset(0)$ refers to the generic $\ppi{0}{\alpha}{X}{\kappa}$ set as usual. Then the second forcing adds a $\ppi{0}{\alpha}{X}{\kappa}$-code for $X \minus G_\emptyset(0)$, in particular $\bP \forces G_\emptyset(0) \in \ssigma{0}{\alpha}{X}{\kappa}$. By the contrapositive of \cref{th: preserve generic pi set} there can be no sufficiently rich family of rank functions on $\bP$. To deal with this problem, one considers only iterations of a more restricted form.
\end{remark}

Let us first summarize some previously known results and types of iterations which have been verified to be ranked forcing notions (for a class of rank functions appropriate for Theorem~\ref{th: preserve generic pi set}):

\begin{itemize}
    \item For $\kappa = \omega$ and 
    \[
        \bP_\gamma \forces \name{A}_\gamma = \name{B}_\gamma = \emptyset, \text{ and } \name{X}_\gamma = \pre{\omega}{2} \cap V^{\bP_\gamma},
    \]
    the proof of the first part of \cite[Theorem~22]{miller_length_1979} verifies that the iteration is a ranked forcing notion. 
    \item The case $\kappa = \omega$,
    \[
        \bP_\gamma \forces \name{X}_\gamma = \check{X}_\gamma
    \]
    where the sequence $\seq{X_\gamma}{\gamma < \gamma^*}$ is in $V$ and $\alpha_\gamma = \alpha_0 + 1$ for $0 < \gamma < \gamma^*$, this is --- up to straightforward adjustments --- the content of \cite[Theorem~34]{miller_length_1979}.
    \item In \cite[Section~7]{gen_borel_sets} we extend the above proof of \cite[Theorem~34]{miller_length_1979} to regular uncountable cardinals $\kappa = \kappa^{<\kappa}$ and iterations with $\bP \forces \name{X}_\gamma = \check{X}$ for an $X \in V$ and $\alpha_\gamma = \alpha_0 + 1$ for $0 < \gamma < \gamma^*$. However, we also assume $\alpha_0 < \omega$ is a finite ordinal.
    
\end{itemize}

Miller's main two results suggest viewing two distinct kinds of iterations, namely
\begin{itemize}
    \item those in which only ground model spaces $\check{X}_\gamma$ appear, hence only ground model spaces can effectively be studied in the forcing extension, and
    \item those in which (proper) names for spaces $\name{X}_\gamma$ appear. Via bookkeeping, such iterations are able to derive statements about all spaces appearing in the final model.
\end{itemize}

Notably, relatively little is known about iterations where (non-ground model) names $\name{X}_\gamma$ appear. The original proof of \cite[Theorem~22]{miller_length_1979} uses only iterands of the form $\aforc_{\alpha_\gamma}(\emptyset, \emptyset, \pre{\omega}{\omega} \cap V^{\bP_\gamma})$, hence it can only provide lower bounds for $\ord_\omega$ - it forces \enquote{upwards} (cf.\ \cref{rem: upwards downwards}). We note that this theorem has a second part, in which for a successor ordinal $2 \leq \beta_0 < \omega_1$ a model is constructed where the spectrum of $\ord_\omega(X)$ has the nontrivial value
\[
    \setOf{\ord_\omega(X)}{X \subseteq \pre{\omega}{2}} = \setOf{\beta}{\beta_0 \leq \beta \leq \omega_1}.
\]
The proof also contains iterands $\aforc_\alpha(A, X \minus A, X)$ that push the hierarchy \enquote{downwards}; however, this only happens on a select sequence of ground model spaces $\seq{X_\alpha}{\beta_0 \leq \alpha < \omega_1} \in V$ consisting of mutually Cohen-generic reals over a smaller model $M \subseteq V$. The properties of Cohen forcing are extensively utilized in order to separate the downwards-pushing aspect of the iteration applied to the ground model spaces $X_\alpha$ from the upwards-pushing aspect of the iteration which, via bookkeeping, ensures $\ord_\omega(X) \geq \beta_0$ holds in the final model for every uncountable $X \subseteq \pre{\omega}{2}$.

On the other hand, as long as one restricts their attention only to ground model spaces, stronger results can be derived. Since one no longer has to work with names for reals $\name{\tau}$ but is instead able to fully decide these names as ground model objects in $\pre{\kappa}{\kappa}$ (cf.\ \cref{prop: good dense}), the existence of rank function is easier to formulate and verify. In such iterations, \cite[Theorem~34]{miller_length_1979} subsumes also the case of downwards-pushing iterands of the form $\aforc_\alpha(A, X \minus A,X)$. However, even here one has to be careful to only admit certain kinds of parameters as to circumvent the counterexample from \cref{rem: no general ranks}. 

The main contribution of this paper is to fully generalize the technique of \cite[Theorem~34]{miller_length_1979} to the case of an uncountable $\kappa = \kappa^{<\kappa}$ while incorporating the idea of \cite[Theorem~52]{miller_length_1979} to provide models for the consistent separation of $\ord_\kappa(X)$ for ground model spaces $X$ by cardinality.

For $\kappa = \omega$, a method for incorporating downwards-pushing iterands in an iteration of $\alpha$-forcing that controls all spaces in the final model has been developed as well and will appear in upcoming work.

\subsection{Iterations with ground model spaces} \label{subsec: alpha forcing ground model}

We work with an iteration as above for which all spaces $X_\gamma$ appearing therein are in the ground model. We simply write $X_\gamma$ instead of $\name{X}_\gamma$. Recall that $A_0 = B_0 = \emptyset$.

\begin{assumption} \label{ass: ground model spaces}
    For each $\gamma < \gamma^*$ we have $\bP_\gamma \forces \name{X}_\gamma = \check{X}_\gamma$ for a space $X_\gamma \in V$.
\end{assumption}

The next is a technical assumption required for some arguments.

\begin{assumption} \label{ass: successor ordinals}
    For each $\gamma < \gamma^*$, if $\bP_\gamma \not \forces \name{A}_\gamma = \name{B}_\gamma = \emptyset$, then $\alpha_\gamma$ is a successor ordinal.
\end{assumption}

Lastly, in order to steer clear of \cref{rem: no general ranks}, we assume the iteration satisfies the following:

\begin{assumption}
    For each $0 < \gamma < \gamma^*$, at least one of the following is the case:
    \begin{itemize}
        \item $\alpha_\gamma > \alpha_0$
        \item $|X_\gamma| < |X_0|$
        \item $\bP_\gamma \forces \name{A}_\gamma = \name{B}_\gamma = \emptyset$.
    \end{itemize}
\end{assumption}

\cref{ass: ground model spaces} enables a much easier analysis of the forcing by passing to a dense subset of conditions of a more tangible form. In several key aspects, the iteration behaves nearly like a product of forcings; for instance, entries in a conditions may be decided as ground model objects. However, since the names $\name{A}_\gamma, \name{B}_\gamma$ do not describe ground model objects, the iteration cannot be treated as a literal product.

\begin{definition}
    A condition $p \in \bP$ is \markdef{good} if for each $\gamma < \gamma^*$, there exist $f_{p(\gamma)}, R_{p(\gamma)} \in V$ such that $p(\gamma)$ is the standard name for the pair
    $\langle \check{f}_{p(\gamma)}, \check{R}_{p(\gamma)} \rangle$.
\end{definition}

Good conditions can be identified with conditions in the product $\prod_{\gamma < \gamma^*} \aforc_{\alpha_\gamma}(\emptyset, \emptyset, X_\gamma)$. Conversely, conditions in the product are also in $\bP$ if they correctly decide membership of the sets $\name{A}_\gamma, \name{B}_\gamma$. Concretely, for a $p \in \prod_{\gamma < \gamma^*} \aforc_{\alpha_\gamma}(\emptyset, \emptyset, X_\gamma)$ we have $p \in \bP$ if and only if for each $\gamma < \gamma^*$, $p \restriction \gamma$ (as a condition in $\bP_\gamma$) forces $x \notin \name{B}_\gamma$ for $x$ with $\langle \emptyset, x \rangle \in R_{p(\gamma)}$ and $y \notin \name{A}_\gamma$ for $y$ with $\langle \eta, y \rangle \in R_{p(\gamma)}, \eta \in \succ(\emptyset)$.

For two good conditions $p$ and $q$, we can denote by $p \cup q$ the pointwise union of the two, i.e.
\[
    (p \cup q)(\gamma) \defas \langle f_{p(\gamma)} \cup f_{q(\gamma)}, R_{p(\gamma)} \cup R_{q(\gamma)} \rangle.
\]
Similarly we define $\bigcup S$ for a set $S$ of good conditions. We warn the reader that even if $p,q$ are compatible, their pointwise union might fail to be a condition. We also write $\points(R_{p(\gamma)})$ for the set $\setOf{x \in X_\gamma}{\exists \eta \in T_{\alpha_\gamma}: \langle \eta, x \rangle \in R_{p(\gamma)}}$.

We first record that although \cref{fact: incomp equiv} no longer holds for the iteration, two variants of it do, in the form of \cref{lem: mixing sequence} and \cref{lem: mixing} (c.f. \cite[Lemma~23]{miller_length_1979}).

\begin{lemma} \label{lem: mixing sequence}
    Suppose $G$ is $\bP$-generic and $\seq{p_i}{i < \delta}, \delta < \kappa$ is a \textbf{decreasing} sequence of good conditions in $G$. Then $\hat{p} \defas \bigcup_{i < \delta} p_i$ is a condition and $\hat{p} \in G$.
\end{lemma}

\begin{proof} \label{prop: good dense}
    A straightforward induction over $\gamma \leq \gamma^*$ using \cref{fact: incomp equiv}.
\end{proof}

\begin{proposition}
    The set of good conditions is dense in $\bP$.
\end{proposition}

\begin{proof}
    Standard. We prove that the set of good conditions is dense in $\bP_\gamma$ for each $\gamma \leq \gamma^*$ by induction. The base case of $\gamma = 0$ is trivial. In the successor case $\gamma \to \gamma + 1$, notice that the forcing $\bP_{\gamma}$ is strategically ${<}\kappa$-closed, and thus no new sequences of elements of $V$ of length ${<}\kappa$ are added. Since both $T_{\alpha_\gamma}$ and, crucially, $X_\gamma$ are in the ground model, for every $p \in \bP_{\gamma+1}$ we have $p \restriction \gamma \forces p(\gamma) \in V$. Thus can find a good condition $p' \leq p \restriction \gamma$ and a pair $\langle f, R \rangle \in V$ such that $p' \forces p(\gamma) = \langle \check{f}, \check{R} \rangle$. The condition $(p' \conc \langle \check{f}, \check{R} \rangle) \leq p$ is therefore good.

    Assume now $\gamma \leq \gamma^*$ is a limit. If $\cf(\gamma) \geq \kappa$, then we are done, since the iteration is with supports of size ${<}\kappa$. So assume $\cf(\gamma) < \kappa$ and let $\seq{\gamma_i}{i < \cf(\gamma)}$ be a continuous cofinal sequence in $\gamma$. For $p \in \bP_\gamma$, we now work in the generic extension and fix a $\bP_\gamma$-generic filter $G$ containing $p$; we construct a decreasing sequence $\seq{p_i}{i < \cf(\gamma)}$ satisfying the following for each $i < \cf(\gamma)$:
    \begin{itemize}
        \item $p_i \in G \restriction \gamma_i$,
        \item $p_i$ is good,
        \item $p_i \leq p \restriction \gamma_i$,
        \item $p_{i+1} \restriction \gamma_i \leq p_i$ and
        \item if $i < \cf(\gamma)$ is a limit ordinal, then $p_i = \bigcup_{i' < i} p_{i'}$.
    \end{itemize}

    The successor step in the construction of the $p_i$'s uses the inductive assumption, and the definition of $p_i$ for $i$ limit in the last bullet point yields a condition by \cref{lem: mixing sequence}. Likewise, the lemma gives us the fact that $p^* = \bigcup_{i < \cf(\delta)} p_i$ is a condition, and it is easy to see that it is both good and stronger than $p$.
\end{proof}

For the rest of this section, all conditions are assumed to be good, even if it is not explicitly stated.

We now move to define the family of putative rank functions on $\bP$\footnote{Technically, the dense subset of $\bP$ consisting of good conditions.}. In the single-step case of the forcing the rank functions carried a parameter $H \subseteq X$; similarly, a sufficiently rich set of admissible parameters can be found in the iteration case.

\begin{definition} \label{def: locus ground}
    We say $H = \seq{H(\gamma)}{\gamma < \gamma^*}$ is a \markdef{locus} if for each $\gamma < \gamma^*$
    \begin{enumerate}
        \item $H(\gamma) \subseteq X_\gamma$,
        \item if $\gamma > 0$, $\alpha_\gamma \leq \alpha_0$, and $\bP_\gamma \not \forces \name{A}_\gamma = \name{B}_\gamma = \emptyset$, then $H(\gamma) = \emptyset$ or $H(\gamma) = X_\gamma$,
        \item for all $x \in H(\gamma)$ the set
        \[
            D = \left\{p \in \bP_\gamma : 
            \begin{gathered}
                \forall \delta < \gamma: \points(R_{p(\delta)}) \subseteq H(\delta) \text{ and} \\
                p \text{ decides } \enquote{x \in \name{A}_\gamma}, \enquote{x \in \name{B}_\gamma}
            \end{gathered}
            \right\}
        \]
        is predense in $\bP_\gamma$, i.e.\ for each $p \in \bP_\gamma$ there exists a $q \in D$ with $q \comp p$.

        Write $\supp(H) = \setOf{\gamma < \gamma^*}{H(\gamma) \neq \emptyset}$. We always assume $0 \in \supp(H(0))$.
    \end{enumerate}
\end{definition}

In the same breath, we introduce the putative notion of rank.

\begin{definition} \label{def: rank ground}
    For a $p \in \bP$ and a locus $H$ define
    \[
        \crank_H(p) \defas \sup_{\gamma < \gamma^*} \sup \setOf{\trank_{T_{\alpha_\gamma}}(\eta)}{\exists x \in X_\gamma \minus H(\gamma): \langle \eta, x \rangle \in R_{p(\gamma)}}
    \]
    if $\supp(p) \subseteq \supp(H)$ and $\crank_H(p) = \infty$ otherwise.
\end{definition}

\begin{remark}
    Proofs will often proceed by induction on the support of a condition $p \in \bP$. As such, we will refer to a notion of goodness, locus and rank relative to $\bP_\gamma$ for $\gamma < \gamma^*$; they are to be defined relative to the partial iteration $\seq{\bP_\delta, \name{\bQ}_\delta}{\delta < \gamma}$. Abusing notation, we write $\crank_H(p)$ for $\crank_{H \restriction \gamma}(p)$.
\end{remark}

\begin{remark}
    In light of Definition~\ref{def: rank ground}, the requirement $\forall \delta < \gamma: \points(R_{p(\delta)}) \subseteq H(\delta)$ in Definition~\ref{def: locus ground} is equivalent to $\crank_H(p \restriction \gamma) = 0$. 
\end{remark}

Working towards applying Theorem~\ref{th: preserve generic pi set}, for a given set of conditions we should be able to find a locus that is not too large on which they vanish.

\begin{lemma} \label{lem: small locus ground}    
Let $\lambda > \kappa$ be a regular cardinal and $F \subseteq \bP$ be a set of size $|F| < \lambda$. If $\lambda > |X_\gamma|$ for all $0 < \gamma < \gamma^*$ with $\alpha_\gamma \leq \alpha_0$ and $\bP_\gamma \not \forces \name{A}_\gamma = \name{B}_\gamma = \emptyset$, then there exists a locus $H$ such that $\crank_H(p) = 0$ for all $p \in F$ and $|H(0)| < \lambda$.
\end{lemma}

\begin{proof}
    Construct an increasing chain $\seq{M_j}{j \leq \omega}$ of elementary submodels of a sufficiently large initial segment of the universe such that $\{\bP, F\} \cup F \cup \kappa \subseteq M_0$ and for all $j < \omega$, $|M_j| < \lambda$ and $M_j \cup \{M_j\} \cup \sup(M_j \cap \lambda) \subseteq M_{j + 1}$. Write $M = \bigcup_{j < \omega} M_j$. Note that for every $u \in M$ with $|u| < \lambda$ we have $u \in M_j$ for some $j < \omega$, hence $|u| \subseteq M_{j+1} \subseteq M$ and thus $u \subseteq M$.
    
    Set $H(\gamma) \defas X_\gamma \cap M$ for $\gamma \in M$ and $H(\gamma) = \emptyset$ otherwise. Clearly $|H(0)| < \lambda$. To see that $H$ is a locus, take $\gamma < \gamma^*$ with $H(\gamma) \neq \emptyset$, in particular $\gamma \in M$. If $\alpha_\gamma \leq \alpha_0$ and $\bP_\gamma \not \forces \name{A}_\gamma = \name{B}_\gamma = \emptyset$, then $\lambda > |X_\gamma|$, so $X_\gamma \subseteq M$. 
    
    To check the third requirement, take $x \in H(\gamma)$. Then the set
    \[
        \setOf{p \in \bP_\gamma}{p \text{ decides } \enquote{x \in \name{A}_\gamma}, \enquote{x \in \name{B}_\gamma}}
    \]
    is dense (and definable in $M$). In $M$, find a maximal antichain $A$ in this dense subset of $\bP_\gamma$; we will show that $\crank_H(p) = 0$ for all $p \in A$. By the $\kappa^+$-c.c., $A \subseteq M$ and for all $\delta \in \supp(p)$, $\points(R_{p(\delta)})$ is definable in $M$. Because of its size it is also a subset of $M$, so $\crank_H(p) = 0$. 
\end{proof}

\begin{remark}
    If we can prove that $\crank_H$ is a rank function for ordinals ${<}\alpha_0$, then Theorem~\ref{th: preserve generic pi set} applies to the iteration and any $\lambda$ satisfying the assumptions of Lemma~\ref{lem: small locus ground}.
\end{remark}

\begin{definition}
    A good condition $p \in \bP$ is \markdef{immaculate} with respect to a locus $H$ if for each $\gamma < \gamma^*$ and $x \in \points(R_{p(\gamma)}) \cap H(\gamma)$ there exists a $t \geq p \restriction \gamma$ so that $\crank_H(t) = 0$ and $t$ decides \enquote{$x \in \name{A}_\gamma$} and \enquote{$x \in \name{B}_\gamma$}.
\end{definition}

The next two lemmas take place in the generic extension.

\begin{lemma} \label{lem: mixing}
    Suppose $G$ is $\bP$-generic and $\seq{p_i}{i < \delta}$ are ${<}\kappa$-many conditions in $G$, each of rank $\crank_H(p_i) \leq \beta < \alpha_0$. Then there exists a $\hat{p} \in G$ with $\hat{p} \leq p_i$ for all $i < \delta$ and $\crank_H(\hat{p}) \leq \beta$.
\end{lemma}

\begin{proof}
    We prove that this is true for every partial iteration $\bP_\gamma$, $\gamma \leq \gamma^*$ by induction on $\gamma$.

    The base case $\gamma = 0$ is trivial, since $\bP_0 = \{\mathbbm{1}\}$. If $\gamma$ is a limit and $\cf(\gamma) \geq \kappa$, then $P \subseteq G \restriction \delta$ for some $\delta < \gamma$, hence we are done. If instead $\cf(\gamma) < \kappa$, fix a cofinal sequence $\seq{\gamma_i}{j < \cf(\gamma)}$ and a $\bP_\gamma$-generic filter $G$. Now inductively construct a sequence $\seq{q_j}{j < \cf(\gamma)}$ such that $q_j \in G \restriction \gamma_j$ is found using this lemma applied to $\bP_{\gamma_j}$ and 
    \[
        \setOf{p_i \restriction \gamma_j}{i < \delta} \cup \setOf{q_{j'}}{j' < j}.
    \]
    The sequence now satisfies $q_{j} \restriction \gamma_{j'} \leq q_{j'}$ for $j' < j$. Using \cref{lem: mixing sequence} and the fact that the sequence is decreasing, one can now see that $\hat{p}$ defined as $\hat{p} \defas \bigcup_{j < \cf(\gamma)} q_j $ has all properties we are looking for.
    
    Let us now focus on the successor case, and assume the statement holds for $\gamma$. We want to show it holds for $\gamma + 1$. Let thus $G$ be $\bP_{\gamma+1}$-generic; write $G = G \restriction \gamma \star G(\gamma)$. As a notational convenience, let us write $f = \bigcup_{i < \delta} f_{p_i(\gamma)}$ and $R = \bigcup_{i < \delta} R_{p_i(\gamma)}$. 

    If $\bP_\gamma \forces \name{A}_\gamma = \name{B}_\gamma = \emptyset$, then we simply let $\hat{p} \restriction \gamma \in G \restriction \gamma$ be a condition constructed inductively from this lemma applied to $\seq{p_i \restriction \gamma}{i < \delta}$ and $\hat{p}(\gamma) = \langle f, R \rangle$. By \cref{fact: incomp equiv}, we are finished.
    
    Otherwise, we know by \cref{ass: successor ordinals} that $\alpha_\gamma$ is a successor ordinal. Let now $Y$ be the set of all $y \in X_\gamma$ for which there exists an $\eta(y) \in \succ_{T_{\alpha_\gamma}}(\emptyset)$ such that for all $\nu \in \succ(\eta(y))$ there is a $\nu'(\nu,y) \in \succ(\nu)$ with $\langle \nu'(\nu,y), y \rangle \in R$. We first note that all such $y$ must be in $H(\gamma)$, because otherwise we would have $\beta \geq \sup_{i < \delta} \crank_H(p_i) \geq \setOf{\trank(\nu'(\nu,y) + 1}{\nu \in \succ(\eta(y))} = \trank(\eta(y))$, which means $\beta + 1 \geq \trank(\eta(y)) + 1 = \alpha_\gamma > \alpha_0$. Since $V[G \restriction \gamma] \models \setOf{p_i(\gamma)}{i < \delta} \subseteq G(\gamma)$, we in particular have $V[G \restriction \gamma] \models \enquote{\eta(y) \text{ is not $y$-critical in some lower bound of the $p_i$'s}}$ and thus $V[G \restriction \gamma] \models y \notin \name{A}_\gamma$. Since we know $Y \subseteq H(\gamma)$, by \cref{def: locus ground} we can find for each $y \in Y$ a $q_y \in G \restriction \gamma$ with $q_y \forces y \notin \name{A}_\gamma$. By induction, we can now apply this lemma to $G_\gamma$ and 
    \[  
    \setOf{p_i \restriction \gamma}{i < \delta} \cup \setOf{q_y}{y \in Y}
    \]
    to get a $q \in G \restriction \gamma$ with $\crank_H(q) \leq \beta$. Now let $\hat{p} \restriction \gamma \defas q$ and $\hat{p}(\gamma) \defas \langle f, R \rangle$. Again following \cref{fact: incomp equiv}, we see that $\hat{p} \restriction \gamma \forces \hat{p}(\gamma) \in \name{\bQ}_\gamma$. The only nontrivial issue that could have arisen is criticality at successors of $\emptyset$, which we have taken care of in the proof.
\end{proof}

\begin{lemma} \label{lem: immaculate dense}
    Let $H$ be a locus, $G$ a $\bP$-generic filter, $p \in G$. Then there exists a $p' \in G$, $p' \leq p$ that is immaculate with respect to $H$ and $\crank_H(p) = \crank_H(p')$.
\end{lemma}

\begin{proof}
    As usual, let $G \restriction \gamma$ for $\gamma < \gamma^*$ denote the $\bP_\gamma$-generic filter produced as a projection of $G$. We will construct a decreasing sequence $\seq{q_j}{j < \omega}$ of conditions in $G$ such that 
    \begin{itemize}
        \item $q_0 = p$,
        \item for every $j < \omega, \gamma < \gamma^*$ and $x \in \points(R_{q_j(\gamma)}) \cap H(\gamma)$ there exists a $t \geq q_{j+1} \restriction \gamma$ such that $\crank_H(t) = 0$ and $t$ decides \enquote{$x \in \name{A}_\gamma$} and \enquote{$x \in \name{B}_\gamma$}.
    \end{itemize}

    To find this sequence, fix $j < \omega$ and for every $\gamma \in \supp(q_j)$, let $K_\gamma = H(\gamma) \cap \points(R_{q_j(\gamma)})$.
    This is a set of size ${<}\kappa$. For each $x \in K_\gamma$ there exists, by Definition~\ref{def: locus ground}, a $q_j^x \in G \restriction \gamma$ with $\crank_H(q_j^x) = 0$ such that $q_j^x$ decides \enquote{$x \in \name{A}_\gamma$} and \enquote{$x \in \name{B}_\gamma$}. Now let $q_{j+1}$ be yielded by \cref{lem: mixing} applied to $G$ and $\setOf{q_j^x}{x \in K_\gamma, \gamma \in \supp(q_j)} \cup \{q_j\}$.

    Now let $p'$ be defined as $p'(\gamma) \defas \bigcup_{j < \omega} q_j$. Since the $q_j$'s are a decreasing sequence, by \cref{lem: mixing sequence} we know $p'$ is a condition and $p' \in G$. A brief moment of contemplation should convince the reader that $p'$ is immaculate. 
\end{proof}

\begin{theorem} \label{th: rank function ground}
    For any locus $H$, $\crank_H$ is a rank function on $\bP$ for ordinals ${<}\alpha_0$ and $h_{\crank_H}(\beta) = \beta+1$.    
\end{theorem}

\begin{proof}
    Fixing an ordinal $\beta < \alpha_0$, and a $p_0 \in \bP$, we produce a condition $q \in \bP, q \comp p_0, \crank_H(q) \leq \beta$ such that for each $r \in \bP, \crank_H(r) < \beta$ we have
    \[
        r \comp q \Rightarrow r \comp p_0.
    \]

    However, instead of constructing $q$ from $p_0$, we will instead first find a stronger condition $p \leq p_0$ and work from there; if we can verify the property for $p$ and $q$, then it is clearly also true for $p_0$ and $q$. First, by Lemma~\ref{lem: immaculate dense} we may assume $p_0$ is immaculate. But this is not the only preparatory step we need to perform; we further strengthen $p_0$ to a $p \leq p_0$, but only in such a way that $\supp(p) = \supp(p_0)$ and $\points(R_{p(\gamma)}) = \points(R_{p_0(\gamma)})$ for all $\gamma < \gamma^*$. This way, $p$ remains immaculate. The condition $p$ should be constructed from $p_0$ by adding ${<}\kappa$-many pairs to $R_{p_0}$ as to witness that the following three properties are satisfied for every $\gamma \in \supp(p)$:
    \begin{itemize}   
        \item For all $\eta \in T_{\alpha_\gamma}$ and $x \in X_\gamma$ where $\trank(\pred(\eta)) > \beta$ is a limit ordinal, $\trank(\eta) < \beta$ and $\langle \pred(\eta), x \rangle \in R_{p(\gamma)}$ there should exist a $\nu \in \succ(\eta)$ with $\langle \nu, x \rangle \in R_{p(\gamma)}$. This can be made to hold: by the construction of the tree $T_{\alpha_\gamma}$ at limit nodes, there are ${<}\kappa$-many of such $\eta, x$ in $R_{p_0(\gamma)}$. Since now $\eta$ has $\kappa$-many successors in $T_{\alpha_\gamma}$, we can find a $\nu \in \succ(\eta)$ so that we can add the tuple $\langle \nu, x \rangle$ to $R_{p(\gamma)}$.
        \item For all $\eta \in T_{\alpha_\gamma}$ such that $p \restriction \gamma \forces \enquote{\eta \text{ is $x$-critical}}$ and $\langle \pred(\eta), x \rangle \in R_{p(\gamma)}$ there is an $\eta' \in \succ(\eta)$ with $\langle \eta', x \rangle \in R_{p(\gamma)}$. This is nearly the same as saying that $p$ should be a strict condition, with the exception being instances of the criticality constraint which arise at successors of the root because of membership of $\name{A}_\gamma$.
        \item The third property is relevant only for those $\gamma$ for which $\alpha_\gamma$ is the successor of a small limit ordinal. Here we consider all $\eta \in \succ_{T_{\alpha_\gamma}}(\emptyset)$ and $x \in X_\gamma$ for which
        \begin{itemize}
            \item $\langle \eta, x \rangle \notin R_{p(\gamma)}$,
            \item there exists a $\nu \in \succ^2(\eta)$ with $\langle \nu, x \rangle \in R_{p(\gamma)}$ and
            \item there also exists an $\eta' \in \succ(\eta)$, $\trank(\eta') \leq \beta$ such that $\langle \nu', x \rangle \notin R_{p(\gamma)}$ for all $\nu' \in \succ(\eta')$.
        \end{itemize}
        For every such $\eta$ and $x$ there should exists an $\eta'$ as above such that $\langle \eta', x\rangle \in R_{p(\gamma)}$.
    \end{itemize}

    There is some nuance here. By Remark~\ref{rem: criticality local}, all limit nodes are vertically spaced out far enough from each other so that by passing through $p_0$ once and collecting all necessary witnesses into $R_p$, no new instances of these properties are created that need to be taken care of, and so we are thereby done in constructing $p \leq p_0$. From this point onward, we simply work with $p$.

    Finally, we define $q$ at each coordinate $\gamma \in \supp(p) \cap \supp(H)$ as
    \[
        q(\gamma) \defas \langle f_{p(\gamma)}, R_{q(\gamma)} \rangle, 
    \]
    where
    \[
        R_{q(\gamma)} \defas \setOf{\langle \eta, x \rangle \in R_{p(\gamma)}}{x \in H(\gamma) \vee \trank_{T_{\alpha_\gamma}}(\eta) \leq \beta}.
    \]
    For $\gamma \notin \supp(\bar{p}) \cap \supp(H)$, we are letting $q(\gamma) = \mathbbm{1}$.

Above all, we must first check that $q$ is actually a condition.

\begin{claim}
    $q$ is a condition in $\bP$.
\end{claim}

\begin{proof}
     If it were not a condition, then there is an index $\gamma$ such that $q \restriction \gamma \in \bP_\gamma$ is a condition, but $q \restriction \gamma \not \forces q(\gamma) \in \name{\bQ}_\gamma$. Looking at the definition of $\alpha$-forcing and the fact that we are dealing with good conditions whose relation $R_{\bar{p}}$ is decided in the ground model, there are only two conceivable ways this could happen. The first of these is if there was an $x \in X_\gamma$ such that either
    \begin{enumerate}[label=\alph*)]
        \item $\langle \emptyset, x \rangle \in R_{q(\gamma)}$ but $q \restriction \gamma \not \forces x \notin \name{B}_\gamma$ or \label{enum: ground th a}
        \item $\langle \eta, x \rangle \in R_{q(\gamma)}$ for an $\eta \in \succ(\emptyset)$ but $q \restriction \gamma \not \forces x \notin \name{A}_\gamma$. \label{enum: ground th b}
    \end{enumerate}
    So let us assume \ref{enum: ground th b} is the case for an $\eta$ and $x$. The case of \ref{enum: ground th a} can be refuted similarly and easier. Since we necessarily now have $\bP_\gamma \not \forces \name{A}_\gamma = \name{B}_\gamma = \emptyset$, we either have $\alpha_\gamma \leq \alpha_0$; in this case, since $\gamma \in \supp(H)$, it follows that $H(\gamma) = X_\gamma$, so $x \in H(\gamma)$. If $\alpha_\gamma > \alpha_0$, then by definition of $R_{q(\gamma)}$ either $x \in H(\gamma)$ or $\trank(\eta) \leq \beta$. But $\alpha_\gamma$ is a successor ordinal, so $\trank(\eta) + 1 = \alpha_\gamma$. So we have $\beta < \alpha_0 \leq \trank(\eta)$. We can conclude that the reason $\langle \eta, x\rangle$ is in $R_{q(\gamma)}$ can only be that $x \in H(\gamma)$.

    So in any case, we know $x \in H(\gamma)$. At this point, the immaculacy of $p$ comes into play. We know that there is a $t \geq p \restriction \gamma$ with $\crank_H(t) = 0$ so that $t$ decides \enquote{$x \in \name{A}_\gamma$} and \enquote{$x \in \name{B}_\gamma$}. Since $\langle \eta, x \rangle \in R_{p(\gamma)}$, $t$ decides \enquote{$x \in \name{A}_\gamma$} and is compatible with $p \restriction \gamma$, this can only mean $t \forces x \notin \name{A}_\gamma$. But now $f_{q \restriction \gamma} = f_{p \restriction \gamma}$ and also every $\langle \nu, y \rangle$ from $R_{p(\delta)}$ with $y \in H(\delta)$ appears in $R_{q(\delta)}$ for $\delta < \gamma$, so actually $t \geq q \restriction \gamma$. This means $q \restriction \gamma \forces x \notin \name{A}_\gamma$, contradiction.

    The second possibility for $q \restriction \gamma \not \forces q(\gamma) \in \name{\bQ}_\gamma$ is that there exists an $\eta \in \succ(\emptyset)$ and for each $\nu \in \succ(\eta)$ a $\nu'(\nu) \in \succ(\nu)$ with $\langle \nu', x \rangle \in R_{q(\gamma)}$ but $q \restriction \gamma \not \forces x \notin \name{A}_\gamma$. Note that the analogous statement for $\eta = \emptyset$ and $q \restriction \gamma \not \forces x \notin \name{B}_\gamma$ is impossible, as $\alpha_\gamma$ cannot be a limit ordinal unless $\bP_\gamma \forces \name{A}_\gamma = \name{B}_\gamma = \emptyset$. Circling back to the situation at hand, we will once again argue that $x \in H(\gamma)$ is implied. If $x \notin H(\gamma)$, then $\beta \geq \crank_H(q) \geq \sup \setOf{\trank(\nu'(\nu))}{\nu \in \succ(\eta)} = \trank(\eta) = \alpha_\gamma - 1 \geq \alpha_0$. This cannot happen, so $x \in H(\gamma)$. But by the immaculacy of $p$, this means there exists a $t \geq p \restriction \gamma$ with $\crank_H(t) = 0$ such that $t$ decides \enquote{$x \in \name{A}_\gamma$} and \enquote{$x \in \name{B}_\gamma$}. Since $R_{q(\gamma)} \subseteq R_{p(\gamma)}$, it must be the case that $t \forces x \notin \name{A}_\gamma$. At this point it is crucial to realize that $t \geq p \restriction \gamma$ and $\crank_H(t) = 0$ implies $t \geq q \restriction \gamma$ by the construction of $q$, so $q \restriction \gamma \forces x \notin \name{A}_\gamma$. This contradicts the assumption of this paragraph.

    We have verified that $q \geq p$ is a condition.
\end{proof}

\begin{claim}
    $q$ has the desired property.
\end{claim}

\begin{proof}
    Suppose $r \in \bP$ is a condition with $\crank_H(r) < \beta$ and $r \comp q$. Without loss of generality, we take $r$ to be immaculate. To see why, suppose $r_0$ is given with $\crank_H(r) < \beta$ and $r_0 \comp q$. Then arguing in the generic extension, let $G$ be a $\bP$-generic filter containing $r_0$ and $q$ and apply Lemma~\ref{lem: immaculate dense} to get an $r \leq r_0$ immaculate with $\crank_H(r) = \crank_H(r_0)$. This $r$ is still compatible with $q$.
    
    We will prove that $p \cup r$ is a condition. For the sake of a contradiction, assume not. This implies the existence of a $\gamma < \gamma^*$ such that $p \restriction \gamma \cup r \restriction \gamma \in \bP_\gamma$ but $p \restriction \gamma \cup r \restriction \gamma \not \forces p(\gamma) \cup r(\gamma) \in \name{\bQ}_\gamma$. There is a list of possible reasons for this to consider. Several of these appear already in the proof of Theorem~\ref{th: rank function single}, and the reasoning for their disqualification may be copied verbatim. By the first five bullet points of the corresponding part of Theorem~\ref{th: rank function single}, we know that $p(\gamma) \cup r(\gamma)$ satisfies condition \ref{def: alpha forcing-c} in Definition~\ref{def: alpha forcing}. Here, the first property we demanded to hold for $p$ comes into play. Furthermore, conditions \ref{def: alpha forcing-d} and \ref{def: alpha forcing-e} follow by the fact that the initial segment of the respective condition forces the requisite statement; for example, for each $x$ with $\langle \emptyset, x \rangle$ we have $p \restriction \gamma \forces x \notin \name{B}_\gamma$ if $\langle \emptyset ,x \rangle \in R_{p(\gamma)}$. The only remaining reason that $p(\gamma) \cup r(\gamma)$ is not a condition is due to failing the criticality constraint. So assume that for a small limit node $\eta \in T_{\alpha_\gamma}$ and every $\eta' \in \succ(\eta)$ there exists a $\nu(\eta') \in \succ(\eta')$ with $\langle \nu(\eta'), x \rangle \in R_{p(\gamma)} \cup R_{r(\gamma)}$ but $p \restriction \gamma \cup r \restriction \gamma \not \forces \enquote{\eta \text{ is not $x$-critical}}$. Here we always choose $\nu(\eta')$ to be in $R_{p(\gamma)}$, if it can be chosen that way. We distinguish cases:
    \begin{itemize}
        \item $\langle \pred(\eta), x \rangle \in R_{p(\gamma)} \cup R_{r(\gamma)}$: This cannot be, by the same proof as in the sixth bullet point in Theorem~\ref{th: rank function single}. We make use of the second property that we demanded for $p$.
        \item $\eta = \emptyset$: This cannot happen, since the possibility that $\eta$ is $x$-critical now entails $\bP_\gamma \not \forces \name{B}_\gamma = \emptyset$, which implies $\alpha_\gamma$ is not a limit ordinal.
        \item $\eta \in \succ(\emptyset)$ and $p \restriction \gamma \cup r \restriction \gamma \not \forces x \notin \name{A}_\gamma$: The interesting case. 
        \begin{itemize}
            \item $x \in H(\gamma)$: Since both $p$ and $r$ are immaculate and $x$ must appear in one of them, there is a condition $t \in \bP_\gamma$ with $\crank_H(t) = 0$ and $t \forces x \in \name{A}_\gamma$ such that either $t \geq p \restriction \gamma$ or $t \geq r \restriction \gamma$. Just like before, we realize that $t \geq p \restriction \gamma$ and $\crank_H(t) = 0$ implies $t \geq q \restriction \gamma$ by the construction of $q$, which overall implies that any condition $s$ stronger than $q$ and $r$ forces $s \restriction \gamma \forces x \in \name{A}_\gamma$. But all pairs $\langle \nu(\eta'), x \rangle$ for $\eta' \in \succ(\eta)$ must appear in $R_{s(\gamma)}$, so overall this situation contradicts the fact that $s$ is a condition.
            \item $x \notin H(\gamma)$: By the assumptions on the iteration and the definition of a locus, we can conclude $\alpha_\gamma > \alpha_0$ in this case. Since $\crank_H(r) < \beta$, $\beta < \alpha_0 < \alpha_\gamma$ and $\alpha_\gamma$ is a successor ordinal, we know $\beta < \trank_{T_{\alpha_\gamma}}(\eta)$. Since $\trank(\eta) = \sup\setOf{\trank(\nu(\eta'))}{\eta' \in \succ(\eta)}$, under the assumption $x \notin H(\gamma)$ this implies that for all $\eta' \in \succ(\eta)$ with $\trank(\eta') \geq \beta + 1$ we have $\langle \nu(\eta'), x \rangle \in R_{p(\gamma)}$. If the rest of the $\langle \nu(\eta'), x \rangle$ are all also in $R_{p(\gamma)}$, then it must mean that $p \restriction \gamma \forces x \notin \name{A}_\gamma$ and we are done. So for one $\eta' \in \succ(\eta)$ with $\trank(\eta') \leq \beta$, we have $\langle \nu(\eta'), x \rangle \notin R_{p(\gamma)}$. Recalling our preparatory step at the beginning of the Theorem, this is exactly the situation we took care to prepare for in the third property we demand for $p$; it follows that in fact, $\langle \eta' ,  x \rangle \in R_{p(\gamma)}$ for some $\eta' \in \succ(\eta)$ with $\trank(\eta') \leq \beta$. This yields a contradiction. \qedhere
        \end{itemize}
    \end{itemize}  
\end{proof}

This concludes the proof of Theorem~\ref{th: rank function ground}.
\end{proof}

\section{Applications of $\alpha$-forcing} \label{sec: applications}

We present several models that are h by using an iteration of $\alpha$-forcing. Although the forcing itself is only strategically ${<}\kappa$-closed, iterations are forcing-equivalent to iterations of the strict partial order $\saforc$, as stated in \cref{cor: strict dense iteration}. Thus we are able write that the forcing extension is ${<}\kappa$-closed into the claims of the theorems.

\begin{theorem} \label{th: set order of set}
    Let $X \subseteq \pre{\kappa}{\kappa}$ be a set of size $|X| > \kappa$ and $1 < \alpha < \kappa^+$ be a successor ordinal. Then there exists a ${<}\kappa$-closed, $\kappa^+$-c.c.\ generic extension $V[G]$ of the universe such that
    \[
        V[G] \models \ord_\kappa(X) = \alpha.
    \]

    The same is also true for $\alpha = \kappa^+$.
\end{theorem}

\begin{proof}
    We first consider the case of a successor $\alpha$, writing $\alpha = \alpha_0 + 1$. Consider the ${<}\kappa$-supported forcing iteration $\langle \bP_\gamma, \name{\bQ}_\gamma : \gamma < 2^\kappa \rangle$, where $\bQ_0 = \aforc_{\alpha_0}(\emptyset, \emptyset, X)$ and for $\gamma > 0$, $\bP_\gamma \forces \name{\bQ}_\gamma = \aforc_{\alpha_0 + 1}(\name{A}_\gamma, X \minus \name{A}_\gamma, X)$. By use of a bookkeeping argument, where we keep in mind that every $\kappa$-Borel subset is coded by an element of $\pre{\kappa}{\kappa}$, we can make sure that every $\kappa$-Borel subset of $X$ in the final model eventually appears as an $\name{A}_\gamma$. This means that every such set is $\ppi{0}{\alpha}{X}{\kappa}$ in $V[G]$ for a $\bP_{2^\kappa}$-generic filter $G$, thus $V[G] \models \ord_\kappa(X) \leq \alpha$.

    To show that $V[G]$, we use Theorem~\ref{th: preserve generic pi set} to get
    \[
    V[G] \models G_\emptyset(0) \notin \ppi{0}{\alpha_0}{X}{\kappa}.
    \]
    To see that the assumptions of the theorem are satisfied, we consider the family 
    \[
        \seq{\crank_H}{\text{ $H$ is a locus with $|H(0)| < \kappa^+$}};
    \]
    by \cref{th: rank function ground}, these are all rank functions for ordinals ${<}\alpha_0$ with $h_{\crank_H}(\beta) = \beta+1$. Furthermore, for every $S \subseteq \bP_{2^\kappa}$, $|S| \leq \kappa$ there exists by \cref{lem: small locus ground} a locus $H$ with $|H(0)| \leq \kappa$ and $\crank_H(p) = 0$ for all $p \in S$. Lastly, we have $\crank_H(p) \geq \crank_{H(0)}(p(0))$.

    Thus \cref{th: preserve generic pi set} takes hold and we get $V[G] \models \ord_\kappa(X) = \alpha$.

    For $\alpha = \kappa^+$, we simply use the forcing iteration $\langle \bP_\gamma, \name{\bQ}_\gamma : \gamma < \kappa^+ \rangle$, where $\bP_\gamma \forces \name{\bQ}_\gamma = \aforc_{\gamma}(\emptyset,\emptyset,X)$. Then in every intermediate model (see the proof of \cref{th: set order of set limit} for a more detailed explanation) \cref{th: preserve generic pi set} holds, which implies that $V^{\bP_{\kappa^+}} \models G_\emptyset(\gamma) \notin \ssigma{0}{\gamma}{X}{\kappa}$ for each $\gamma < \kappa^+$. Hence we have $V^{\bP_{\kappa^+}} \models \ord_\kappa(X) = \kappa^+$.
\end{proof}

One might wonder about setting the order of a space to be a limit ordinal other than $\kappa^+$. In this direction, we have a partial result. Since it appears that we are not able to control the limit levels of the Borel hierarchy directly (cf.\ \cref{ass: successor ordinals}), we have to resort to doing so indirectly by decomposing the space into smaller pieces.

\begin{theorem} \label{th: set order of set limit}
    Suppose $X \subseteq \pre{\kappa}{\kappa}$ is a space whose size $|X|$ is a limit cardinal of cofinality $\lambda$, $|X| > \kappa$. Furthermore let $\zeta < \kappa^+$ be a limit ordinal of cofinality $\lambda$. Then there exists a ${<}\kappa$-closed, $\kappa^+$-c.c.\ generic extension $V[G]$ of the universe such that
    \[
        V[G] \models \ord_\kappa(X) = \zeta.
    \] 
\end{theorem}

\begin{proof}
    Let us first fix a cofinal sequence $\seq{\zeta_i}{i < \lambda}$ in $\zeta$, as well as a strictly increasing sequence $\seq{\Theta_i}{i < \lambda}$ consisting of cardinals cofinal in $\Theta = |X|$, with $\Theta_0 > \kappa$. Furthermore, we choose an increasing sequence $\seq{\Xi_i}{i < \lambda}$ of subsets of $X$ such that $X = \bigcup_{i < \lambda} \Xi_i$ and $|\Xi_i| = \Theta_i^{+}$. By first forcing with a preparatory partial order ($\alpha$-forcing for $\alpha = 2$), we may assume that $\Xi_i \in \ssigma{0}{2}{X}{\kappa}$ for each $i < \lambda$.

    The iteration of $\alpha$-forcing used to construct the desired model is $\langle \bP_\gamma, \name{\bQ}_\gamma: \gamma < 2^\kappa \rangle$, where 
    \begin{itemize}
        \item $\bP_i \forces \name{\bQ}_i = \aforc_{\zeta_i}(\emptyset, \emptyset, \Xi_i)$ for $i < \lambda$
        \item for every $\gamma$ with $\lambda \leq \gamma < 2^\kappa$ there exists a $j(\gamma) < \lambda$ such that 
        \[
            \bP_\gamma \forces \name{\bQ}_\gamma = \aforc_{\zeta_{j(\gamma)} + 1}(\name{A}_\gamma, \Xi_{j(\gamma)} \minus \name{A}_\gamma, \Xi_{j(\gamma)})
        \]
        \item by bookkeeping, we ensure that for each $i < \lambda$, every $\kappa$-Borel subset of $\Xi_i$ in the final model appears as an $\name{A}_\gamma$ as above, for a $\gamma$ with $j(\gamma) = i$. 
    \end{itemize}
    The last bullet point ensures that $\ord_\kappa(\Xi_i) \leq \zeta_i + 1$ for each $i < \lambda$. To see that $\ord_\kappa(\Xi_i) = \zeta_i + 1$, we once again invoke \cref{th: preserve generic pi set}, which will give us $\bP_{2^\kappa} \forces G_\emptyset(i) \notin \ssigma{0}{\zeta_i}{\Xi_i}{\kappa}$ for each $i < \lambda$. Here $G_\emptyset(i)$ is the generic $\ppi{0}{\zeta_i}{\Xi_i}{\kappa}$ set added by the $i$-th generic filter in the iteration. From a technical perspective, to use the theorem in this manner we are working in the extension $V^{\bP_i}$ and applying it to the tail iteration $\bP_{i, 2^\kappa} = \langle \bP_\gamma, \name{\bQ}_\gamma: i \leq \gamma < 2^\kappa \rangle$. 
    
    \begin{claim}
        For every $i < \lambda$ we have $\bP_i \forces \bP_{i, 2^\kappa} \forces G_\emptyset(i) \notin \ssigma{0}{\zeta_i}{\Xi_i}{\kappa}$.
    \end{claim}
    \begin{proof}[Proof of the claim]
        We work in $V^{\bP_i}$ and check that \cref{th: preserve generic pi set} is in fact applicable for the cardinal $\Theta_i$. As before, we do this by claiming that
        \[
            \seq{\crank_H}{\text{ $H$ is a locus with $|H(0)| \leq \Theta_i$}}
        \]
        witnesses the assumptions of the theorem. We only prove that \cref{enum: switcheroo 2} is satisfied, the rest is the same as in \cref{th: set order of set}. But if $\gamma > i$ with $\alpha_\gamma \leq \alpha_i$ and $\bP_\gamma \not \forces \name{A}_\gamma = \name{B}_\gamma = \emptyset$, then $j(\gamma) \leq i$ and hence $|X_\gamma| = \Theta_{j(\gamma)}^{+} \leq \Theta_i$. Thus \cref{lem: small locus ground} (applied to the cardinal $\Theta_i^{+}$) implies that \cref{enum: switcheroo 2} in \cref{th: preserve generic pi set} is satisfied. Since $|\Xi_i| = \Theta_i^{+}$, we thus in fact even have
        \[
            V^{\bP_i} \models \bP_{i, 2^\kappa} \forces G_\emptyset(i) \notin \ssigma{0}{\zeta_i}{\Xi_i}{\Theta_i},
        \]
        so in particular the claim is true.
    \end{proof}

    This means that for every $i < \lambda$ we have
    \[
        \bP_{2^\kappa} \forces \ord_\kappa(\Xi_i) = \zeta_i + 1.
    \]
    Since $\cf(\zeta) = \lambda$, we have $\lambda \leq \kappa$. Hence if $B \subseteq X$ is a $\kappa$-Borel subset of $X$ in the final model, then for every $i < \lambda$, we have $V^{\bP_{2^\kappa}} \models B \in \ssigma{0}{\zeta_i + 1}{\Xi_i}{\kappa}$. This immediately yields $V^{\bP_{2^\kappa}} \models B \in \ssigma{0}{\zeta}{X}{\kappa}$, and since $\ord_\kappa(X) \geq \sup_{i < \lambda} \ord_\kappa(\Xi_i)$, we get
    \[
        V^{\bP_{2^\kappa}} \models \ord_\kappa(X) = \zeta. \qedhere
    \]
\end{proof}

\begin{remark}
    In fact, by \cite[Proposition~4.4]{gen_borel_sets}, and particularly \cite[Claim~4.4.1]{gen_borel_sets}, every space $X \subseteq \pre{\kappa}{\kappa}$ whose order is a limit ordinal $\zeta$ has a similar form to the one constructed in \cref{th: set order of set limit}; there exists a sequence of disjoint closed subspaces whose orders converge to $\zeta$.
\end{remark}

Other variations of \cref{th: set order of set} are possible by using a longer iteration of $\alpha$-forcing. By fixing a regular cardinal $\lambda \geq \kappa$ and enumerating every $\lambda$-Borel subset of $X$ in an iteration of $\alpha$-forcing of cofinality ${>}\lambda$, we get the consistency of $\ord_\lambda(X) = \alpha$ for any successor ordinal $\alpha$ and $X \subseteq \pre{\kappa}{\kappa}, |X| > \lambda$. This is because of the stronger conclusion given by \cref{th: preserve generic pi set}.
    
If we instead iterate up to $\gamma^* = 2^{|X|}$, with a bookkeeping argument we are able to catch \textit{every} subset of $X$ as an $\name{A}_\gamma$; this means that not only do we set $\ord_\kappa(X)$ to a desired ordinal, but we also get $\powset(X) = \ssigma{0}{\alpha}{X}{\kappa}$, Using \cref{th: preserve generic pi set} and previously employed arguments, we get a characterization for a consistent assignment of $\ord_\lambda$ for several $\lambda$ simultaneously. We record this observation in the form of a corollary.

\begin{corollary} \label{cor: single set make all borel}
    Suppose $X \subseteq \pre{\kappa}{\kappa}$ and $\alpha$ is a successor ordinal. Then there is a ${<}\kappa$-closed, $\kappa^+$-c.c.\ generic extension $V[G]$ in which $\powset(X) = \ssigma{0}{\alpha}{X}{\kappa}$ and for infinite cardinals $\lambda$,
    \[
        \ord_\lambda(X) =
        \begin{cases}
            1 &\text{ if $X$ is discrete} \\
            2 &\text{ if $X$ is not discrete and $|X| \leq \lambda$} \\
            \alpha &\text{ if $X$ is not discrete and $\kappa \leq \lambda < |X|$}.
        \end{cases}
    \]
\end{corollary}

\subsection{Setting $\ord_\kappa$ for many spaces simultaneously}

Given the scope and generality of the iteration investigated in Section~\ref{subsec: alpha forcing ground model}, we can even exert much more refined control over the $\kappa$-Borel structure on $\pre{\kappa}{\kappa}$. In this section, assignments of $\ord_\kappa$ for several spaces $X$ at once are considered. While we continue being limited by the fact that all spaces modified throughout the constructed must be already contained in the ground model, it will turn out that we can consistently separate spaces by their size, i.e.\ construct a model of $\ord_\kappa(X) < \ord_\kappa(Y)$ if $|X| < |Y|$. The core argument for this separation has already appeared in \cref{th: set order of set limit}.

\begin{theorem} \label{th: set order of all spaces}
    Let $f$ be a function assigning to each cardinal $\lambda$ with $\kappa < \lambda \leq 2^\kappa$ an ordinal $1 < f(\lambda) \leq \kappa^+$ such that
    \begin{enumerate}
        \item $f$ is (not necessarily strictly) increasing,
        \item if $\lambda$ is a successor cardinal, then $f(\lambda)$ is a successor ordinal or $f(\lambda) = \kappa^+$,
        \item Continuity: if $\lambda$ is a limit cardinal, then $f(\lambda) = \sup_{\lambda' < \lambda} f(\lambda')$.
    \end{enumerate}
    Then there exists a ${<}\kappa$-closed, $\kappa^+$-c.c.\ generic extension $V[G]$ of the universe such that
    \[
        V[G] \models \forall X \subseteq \pre{\kappa}{\kappa}, X \in V, |X| > \kappa: \ord_\kappa(X) = f(|X|).
    \]
\end{theorem}

\begin{proof}
    For every $X \subseteq \pre{\kappa}{\kappa}$ for which $|X|$ is a limit cardinal and $f(|X|)$ is a limit ordinal\footnote{$f(|X|)$ can be a successor ordinal, which is the case of and only if $f$ is eventually constant below $|X|$ with value ${<}\kappa^+$.}, we first fix a cofinal sequence $\seq{\zeta_i^X}{i < \cf(|X|)}$ in $|X|$ consisting of successor cardinals. Notice that now $\seq{f(\zeta_i^X)}{i < \cf(|X|)}$ is a sequence of successor ordinals cofinal in $f(|X|)$ by the continuity of $f$. Lastly, fix an $\subseteq$-increasing sequence $\seq{\Xi_i^X}{i < \cf(|X|)}$ of subsets of $X$ with $|\Xi_i^X| = \zeta_i$. The first step in our construction is going to be a preparatory forcing, namely the generic extension by a iteration of $\alpha$-forcing with the property that for each $X$ as above and $i < \cf(|X|)$, the forcing $\aforc_2(\Xi_i^X, X \minus \Xi_i^X, X)$ appears. Hence we are allowed to assume, without loss of generality, that the ground model $V$ already satisfies $V \models \Xi_i^X \in \ssigma{0}{2}{X}{\kappa}$ for every such $X,i$.

    For the main construction of the forcing, we take an iteration of $\alpha$-forcing $\bP = \langle \bP_\gamma, \name{\bQ}_\gamma : \gamma < \gamma^* \rangle$ of sufficient length such that the following, and \textit{only} the following, forcings appear:
    \begin{enumerate}
        \item Enumerate all subsets of $\pre{\kappa}{\kappa}$ in the ground model as $\seq{X_j}{j < 2^{2^\kappa}}$. Then the first $2^{2^\kappa}$-many iterands of $\bP$ should look as follows:
        \begin{itemize}
            \item If $f(|X_j|) = f(|X_j|)^- + 1$ is a successor ordinal, then there is a $\gamma < 2^{2^\kappa}$ with $\bP_\gamma \forces \name{\bQ}_\gamma = \aforc_{f(|X_j|)^-}(\emptyset, \emptyset, X_j)$.
            \item If $f(|X_j|) = \kappa^+$, then for every $\alpha < \kappa^+$ there is a $\gamma < 2^{2^\kappa}$ with $\bP_\gamma \forces \name{\bQ}_\gamma = \aforc_{\alpha}(\emptyset, \emptyset, X_j)$.
            \item If $f(|X_j|) < \kappa^+$ is a limit ordinal, then it follows that $|X_j|$ is a limit cardinal. For this case we put no requirements on the iteration.
        \end{itemize}
        \item For iterands $\gamma \geq 2^{2^\kappa}$ there should exist a $Y(\gamma) \subseteq \pre{\kappa}{\kappa}$ such that $f(|Y(\gamma)|)$ is a successor ordinal and $\bP_\gamma \forces \name{\bQ}_\gamma = \aforc_{f(|Y(\gamma)|)}(\name{A}_\gamma, Y(\gamma) \minus \name{A}_\gamma, Y(\gamma))$.
        \item By bookkeeping, ensure that for every $X \in V$ where $f(|X|)$ is a successor ordinal and for every name $\name{A}$ for a $\kappa$-Borel subset of $X$ in the final model there exists a $\gamma < \gamma^*$ such that $Y(\gamma) = X$ and $\bP_\gamma \forces \name{A}_\gamma = \name{A}$.
    \end{enumerate}

    Firstly, let us verify the upper bounds claimed in the statement of the theorem. Let thus $X \subseteq \pre{\kappa}{\kappa}$ be given. If $f(|X|)$ is a successor ordinal, then the third clause above guarantees that every $\kappa$-Borel subset of $X$ in $V^\bP$ appears as the parameter for one of the forcings $\aforc_{f(|X|)}(\name{A}, X \minus \name{A}, X)$ in the iteration; hence in the final model, we have $V^\bP \models \forall A \subseteq X, A \in \bor{X}{\kappa}: A \in \ssigma{0}{f(|X|)}{X}{\kappa}$, which entails $V^\bP \models \ord_\kappa(X) \leq f(|X|)$. If $f(|X|) = \kappa^+$, there is nothing to prove. Lastly, if $f(|X|) < \kappa^+$ is a limit, then $|X|$ is a limit cardinal. From the beginning of the proof, we recover the sequence $\seq{\Xi_i^X}{i < \cf(|X|)}$ and notice that since $|\Xi_i^X|$ are all successor cardinals, we have $V^\bP \models \ord_\kappa(\Xi_i^X) \leq f(|\Xi_i^X|) = f(\zeta_i^X)$. In particular, for every $A \subseteq X$, $A \in V^\bP$ and $i < \cf(|X|)$ we have $V^\bP \models A \cap \Xi_i^X \in \ssigma{0}{f(\zeta_i^X)}{\Xi_i^X}{\kappa}$. Since we also made sure that $V^\bP \models \Xi_i^X \in \ssigma{0}{2}{X}{\kappa}$, we have
    \[
        V^\bP \models A = \bigcup_{i < \cf(|X|)} [A \cap \Xi_i^X] \in \ssigma{0}{ \sup_{i < \cf(|X|)} f(\zeta_i^X)}{X}{\kappa} = \ssigma{0}{f(|X|)}{X}{\kappa}.
    \]
    Note that $\cf(|X|) \leq \kappa$ must be true here since $\cf(f(|X|)) \leq \kappa$. Hence $V^\bP \models \ord_\kappa(X) \leq f(|X|)$ as well.

    We now turn to the lower bounds. It will be enough for us to verify the lower bound only for $X$ for which $|X|$ is a successor cardinal. To see why, we firstly remark that for spaces $X$ for which $f(|X|) < \kappa^+$ is a limit, the statement $V^\bP \models \ord_\kappa(X) = f(|X|)$ will follow immediately from the fact that $V^\bP \models \ord_\kappa(\Xi_i^X) = f(\zeta_i^X)$, $i < \cf(|X|)$. Likewise, if the first cardinal $\Theta$ for which $f(\Theta) = \kappa^+$ is a limit cardinal, then $\cf(\Theta) = \kappa^+$, so for any $X$ with $|X| = \Theta$ there is a sequence of subspaces of $X$ whose order converges to $\kappa^+$ by the continuity of $f$, hence $\ord_\kappa(X) = \kappa^+$ by the monotonicity of $\ord_\kappa$.

    Let thus $X \subseteq \pre{\kappa}{\kappa}$ be given whose size is a successor cardinal $|X| = (|X|^-)^+$. If\footnote{Here we employ the nonstandard yet evocative notation of letting $\theta^-$ denote the preceding cardinal if $\theta$ is a successor cardinal and the preceding ordinal if $\theta$ is a successor ordinal.} $f(X) = f(|X|)^- + 1$ is a successor, then there is a $\gamma < 2^{2^\kappa}$ with $\bP_\gamma \forces \name{\bQ}_\gamma = \aforc_{f(|X|)^-}(\emptyset, \emptyset, X)$. Once again, working in the $\gamma$-th model, we claim that \cref{th: preserve generic pi set} applies with $\lambda = |X|^-$. To see why, it suffices to know that \cref{lem: small locus ground} holds for $\lambda = |X|$; but if $\gamma' > \gamma$ with $\alpha_{\gamma'} \leq f(|X|)^-$ in this iteration, then it must be the case that 
    \[
    \bP_{\gamma'} \forces \name{\bQ}_{\gamma'} = \aforc_{f(|Y(\gamma')|)}(\name{A}_{\gamma'}, Y(\gamma') \minus \name{A}_{\gamma'}, Y(\gamma'))
    \]
    where $|Y(\gamma')| < |X|$ by the monotonicity of $f$. Hence we are done, \cref{th: preserve generic pi set} applies and $V^\bP \models \ord_\kappa(X) = f(|X|)$. The remaining case that $|X|$ is a successor cardinal but $f(|X|) = \kappa^+$ can be proven the same way, where we now use the fact that in the first part of the iteration we added generics for $\aforc_\alpha(\emptyset,\emptyset,X)$ for every $\alpha < \kappa^+$.

    This concludes the proof of the theorem.
\end{proof}

\begin{remark}
    The continuity requirement on $f$ in \cref{th: set order of all spaces} is necessary. Suppose we have an increasing sequence $\seq{\Theta_i}{i < \delta}, \delta \leq \kappa$ of cardinals, a corresponding increasing sequence $\seq{\zeta_i}{i < \delta}$ of ordinals and (any, not necessarily a forcing extension) an outer model $V' \supseteq V$ of the universe satisfying
    \[
    V' \models \forall X \subseteq \pre{\kappa}{\kappa}, X \in V: |X| = \Theta_i \Rightarrow \ord_\kappa(X) = \zeta_i.
    \]
    Then consider any $\seq{X_i}{i < \delta}$ with $|X_i| = \Theta_i$ which are clopen separated, i.e.\ for any $i$ there is a clopen set $C$ with $X_i \subseteq C$ such that $X_j \cap C = \emptyset$ for $j \neq i$. Then for $X = \bigcup_{i < \delta} X_i$ we know $X_i \in \ssigma{0}{1}{X}{\kappa}$, hence $V' \models \ord_\kappa(X) = \sup_{i < \delta} \zeta_i$ by familiar arguments.
\end{remark}

\begin{remark}
    To have a model of $\alpha_1 = \ord_\kappa(X) < \ord_\kappa(Y) = \alpha_2$, it is at least necessary that $Y$ does not embed into $X$ via a continuous mapping (or even a $\kappa$-Borel measurable mapping of low complexity), and moreover that this property is preserved by the forcing employed. If $|X| < |Y|$, then this is surely the case, but we do not exclude the possibility that a finer analysis of the ranked forcing properties of $\alpha$-forcing could yields a less restrictive requirement. As it stands, the only case of separating the order of two spaces which does not rely on differing cardinalities is in the second part of \cite[Theorem~22]{miller_length_1979}, in which this was done specifically for a set of mutually Cohen generic reals. It does not seem that these arguments can be applied more broadly.
\end{remark}

\subsection{Preserving definability}

It is a ubiquitous fact in descriptive set theory that every Borel subset $B$ of the reals satisfies the \textit{perfect set property}, namely that it is either countable or there exists a continuous embedding $\pre{\omega}{2} \hookrightarrow B$. When passing to the study of generalized descriptive set theory of an uncountable cardinal, one quickly notices that many topological dichotomy theorems no longer hold sway; this is the case with the perfect set property as well, and it is easy to (consistently) construct even closed counterexamples, such as $\kappa$-Kurepa trees.

An argument of Silver \cite{silver} shows that whenever one is given a tree $T \subseteq \pre{<\omega_1}{\omega_1}$, then this tree does not gain any new branches in $\sigma$-closed forcing extensions unless it had a perfect subtree in the ground model. This was later extended by L\"ucke to the class of analytic sets \cite[Lemma~7.6]{lucke_sigma_2012} and adapted to the case of $\kappa$-Borel codes in \cite{gen_borel_sets}. For a $\kappa$-Borel set $X \subseteq \pre{\kappa}{\kappa}$ and an outer model $V' \supseteq V$ of the universe, the notation $X^{(V')}$ refers to the reinterpretation of a\footnote{This is sensible mostly when $V$ is a $\ssigma{1}{1}{}{\kappa}$-elementary submodel of $V'$, in which case the choice of code does not matter.} $\kappa$-Borel code for $X$ as a $\kappa$-Borel subset of $\pre{\kappa}{\kappa} \cap V'$ in the outer model.

\begin{proposition}[\protect{\cite[Lemma~7.6]{lucke_sigma_2012}}, \protect{\cite[Corollary~6.10]{gen_borel_sets}}] \label{prop: perfect no new elems}
    Let $X \subseteq \pre{\kappa}{\kappa}$ be $\kappa$-Borel. Then the following are equivalent:
    \begin{enumerate}
        \item Every ${<}\kappa$-closed forcing $\bP$ which adds a new subset of $\kappa$ adds a new element to $X$, i.e.\ $\bP \forces X^{(V)} \subsetneq X^{(V^\bP)}$.
        \item There exists a ${<}\kappa$-closed forcing which adds a new element to $X$.
        \item $X$ has a perfect subset in $V$: there exists a continuous embedding of $\pre{\kappa}{2} \hookrightarrow X$ with closed image.
        \item There exists a $\kappa$-Borel measurable\footnote{A function is $\kappa$-Borel measurable if the preimage of every open set is $\kappa$-Borel.} injection of $\pre{\kappa}{2}$ into $X$.
    \end{enumerate}
\end{proposition}

Since every model in this section is constructed as a ${<}\kappa$-closed forcing extension of the ground model, they can all be decorated with the additional clause that if any such space $X$ is $\kappa$-Borel definable as a subset of $\pre{\kappa}{\kappa} \cap V$ and $X$ does not contain a perfect subset in $V$, then the conclusion of the respective theorem holds and additionally the same $\kappa$-Borel code also defines $X$ as a subset of $\pre{\kappa}{\kappa} \cap V^{\bP}$ in the extension. For example, we get a variant of \cref{th: set order of set}:

\begin{corollary}
    Let $X \subseteq \pre{\kappa}{\kappa}$ be a set of size $|X| > \kappa$ and $1 < \alpha \leq \kappa^+$ be either a successor ordinal or $\alpha = \kappa^+$. Assume $X$ is a $\kappa$-Borel subset of $\pre{\kappa}{\kappa}$ in $V$ no perfect subset.
    Then there is a ${<}\kappa$-closed, $\kappa^+$-c.c.\ forcing extension which satisfies 
    \[
    V[G] \models \ord_\kappa(X^{(V)}) = \alpha \text{ and } X^{(V)} = X^{(V[G])}.
    \]
\end{corollary}

The following is a result of Hamkins; we cite it in the form presented in \cite[Theorem~6.12]{gen_borel_sets}.

\begin{theorem}[\protect{\cite[Key~Lemma]{hamkinsGapForcing2001}}] \label{th: hamkins}
    Let $\bP$ be a nontrivial forcing notion of size $|\bP| < \kappa$. Then for any closed set $X \subseteq \pre{\kappa}{\kappa}$ in the ground model,
    \[
        V^\bP \models X^{(V)} \text{ is closed with no perfect subset}.
    \]
\end{theorem}

Consider the generic extension of $V$ obtained by first forcing with a partial order of size ${<}\kappa$ and then following up with an iteration of $\aforc_2(X, (\pre{\kappa}{\kappa} \cap V) \minus X, \pre{\kappa}{\kappa} \cap V)$ where every $X \subseteq \pre{\kappa}{\kappa}, X \in V$ appears as a parameter; similarly to the extension used in \cref{cor: single set make all borel}. Then after the first forcing, $\pre{\kappa}{\kappa} \cap V$ is $\ppi{0}{1}{\pre{\kappa}{\kappa}}{\kappa}$ by \cref{th: hamkins}, and by \cref{prop: perfect no new elems} this is also true in the final model. Hence the final model will satisfy $X \in \ssigma{0}{2}{\pre{\kappa}{\kappa}}{\kappa}$ for all $X \subseteq \pre{\kappa}{\kappa}, X \in V$. Afterwards, any of the forcing iterations from theorems in this section may be applied to change the values of $\ord_\kappa(X)$ for $X \in V$.

As an example, we formulate a variant of \cref{th: set order of all spaces}.

\begin{corollary}
    Suppose $f$ is a function satisfying the properties laid out in \cref{th: set order of all spaces} and $\bP$ is a nontrivial forcing of size ${<}\kappa$. Then there exists a $\bP$-name $\name{\bQ}$ for a ${<}\kappa$-closed, $\kappa^+$-c.c.\ forcing such that 
    \[
        V^{\bP \star \name{\bQ}} \models \forall X \subseteq \pre{\kappa}{\kappa}, X \in V, |X| > \kappa: \ord_\kappa(X) = f(|X|) \wedge X \in \ssigma{0}{2}{\pre{\kappa}{\kappa}}{\kappa}.
    \]
\end{corollary}

We also formulate an extension of our previous joint result with Agostini, Motto Ros and Pitton.

\begin{corollary}
    For any $0 < \alpha \leq \kappa^+$ there consistently exists a closed subset $X \subseteq \pre{\kappa}{\kappa}$ with $\ord_\kappa(X) = \alpha$.
\end{corollary}

\begin{proof}
    For $\alpha = 1$, take any $X$ with $|X| = 1$; likewise, for $\alpha = \kappa^+$ the whole space $X = \pre{\kappa}{\kappa}$ suffices.

    For successor $\alpha$, first force with a forcing of size ${<}\kappa$ to turn $\pre{\kappa}{\kappa} \cap V$ into a closed set with no perfect subset, and then apply \cref{th: set order of set}. Then $\ord_\kappa(\pre{\kappa}{\kappa} \cap V) = \alpha$ in the extension, and by \cref{prop: perfect no new elems} it remains closed.

    For $\alpha < \kappa^+$ limit it is enough to construct a closed set $X \subseteq \pre{\kappa}{\kappa}$, where $|X| > \kappa$ is a limit cardinal with $\cf(|X|) = \cf(\alpha)$. Afterwards, we can apply \cref{th: hamkins} and \cref{th: set order of set limit} to get the stated conclusion. Let us thus take fix any cardinal $\Lambda > \kappa$ of cofinality $\cf(\Lambda) = \cf(\alpha)$. One possible way to construct a model with a closed set $X \subseteq \pre{\kappa}{\kappa}$ of size $\Lambda$ is as follows: work of Holy and L\"ucke (see for example \cite[Definition~2.1,Theorem~2.3]{holy_simplest_2017}) implies that there exists a forcing $\bK$ which adds a tree $T \subseteq \pre{<\kappa}{\kappa}$ such that
    \[
        V^{\bK} \models [T] \text{ has no perfect subset}
    \]
    and\footnote{To see this, note that the trees $\dot{T}_\alpha^G$ constructed in \cite[Theorem~2.3]{holy_simplest_2017} for $A = \pre{\kappa}{\kappa} \cap V$ are all the same up to ${<}\kappa$-many modifications, hence all of the same size.} $V^{\bK} \models |[T]| = 2^\kappa = \left(2^{\kappa}\right)^V$. The forcing is ${<\kappa}$-closed and has a strong version of the $\kappa^+$-c.c., in particular it satisfies $(\heartsuit)$ and is iterable by \cref{prop: heart cc}. Hence to construct our desired model, start with a model of $2^\kappa = \kappa^+$ and in a ${<}\kappa$-supported $\cf(\Lambda)$-long iteration alternate between the forcing $\bK$ and a forcing increasing $2^\kappa$. Since the generic trees $T$ added by $\bK$ have no perfect subset, this method allows us to add a sequence $\seq{T_i}{i < \cf(\Lambda)}$ of trees such that $\seq{|[T_i]|}{i < \cf(\Lambda)}$ is a strictly increasing sequence of cardinals approaching $\Lambda$. Then in the final model, $T^* = \bigcup_{i < \cf(\Lambda)} T_i$ is a tree with $|[T^*]| = \Lambda$.
\end{proof}

\section{$\kappa$-Steel forcing} \label{sec: steel}

In this section, we consider a possible generalization of Steel's forcing with tagged trees \cite{steelForcingTaggedTrees1978} and show that it fits the framework of ranked forcing. As an application, we calculate the $\kappa$-Borel complexity of certain classes of well-founded trees.

We first define a dense subset of the $\kappa$-Steel forcing.

\begin{definition}
    Let $\alpha \leq \kappa^+$ be an ordinal. Define the forcing $\bS_\alpha'$ consisting of all pairs $p = \langle t_p, \rho_p \rangle$ such that
    \begin{itemize}
        \item $t_p \subseteq \pre{<\omega}{\kappa}$ is a well-founded tree of size ${<}\kappa$,
        \item $\rho_p : t_p \to \alpha$ is a function such that for all $\eta \subsetneq \nu$ we have $\rho_p(\eta) > \rho_p(\nu)$.
    \end{itemize}
    The ordering is given by $q \leq p$ if and only if $t_p \subseteq t_q$ and $\rho_p \subseteq \rho_q$.
\end{definition}

The function $\rho_p$ is called a \textit{tagging} of $t_p$. Every generic filter $G \subseteq \bS_\alpha'$ defines a pair $\langle t_G, \rho_G\rangle \defas \bigcup_{p \in G} \langle t_p, \rho_p\rangle$. By density, it is not hard to see that $t_G \subseteq \pre{<\omega}{\kappa}$ is a well-founded tree and $\rho_G : t_G \to \alpha$ assigns to each node its well-founded rank in $t_G$.

This definition is in line with usual treatments of Steel's forcing for $\kappa = \omega$. However, for the purposes of introducing a ranking function, it is more convenient to work with an adjusted partial order which allows for partial taggings and contains $\bS_\alpha'$ as a dense subset. One change from the countable setting is the introduction of ${<}\kappa$-closure as a desirable regularity property for the forcing; this causes tension with the fact that well-foundedness is intrinsically tied to $\omega$. As a consequence, and given the fact that our application involves coding Borel codes, we eschew the possibility of an $\infty$-tag, which classically is used to control the non-well-founded part of the generic tree. A different generalization of the forcing consisting of non-well-founded trees appears in work of S. D. Friedman \cite{friedman}.

\begin{definition}
    Let $\alpha \leq \kappa^+$ be an ordinal. The \textit{$\kappa$-Steel forcing} $\bS_\alpha$ consist of all triples $p = \langle t_p, \rho_p, \bar{\rho}_p \rangle$ such that
    \begin{itemize}
        \item $t_p \subseteq \pre{<\omega}{\kappa}$ is a well-founded tree of size ${<}\kappa$,
        \item $\rho_p, \bar{\rho}_p : t_p \to \alpha$ are partial functions with $
        \dom(\rho_p) \cup \dom(\bar{\rho}_p) = t_p$ and $\dom(\rho_p) \cap \dom(\bar{\rho}_p) = \emptyset$,
        \item for each pair $\eta, \nu$ with $\nu \in \succ(\eta)$ one of the following is true:
        \begin{itemize}
            \item $\{\nu, \eta\} \subseteq \dom(\rho_p)$ and $\rho_p(\nu) < \rho_p(\eta)$,
            \item $\{\nu, \eta\} \subseteq \dom(\bar{\rho}_p)$ and $\bar{\rho}_p(\nu) < \bar{\rho}_p(\eta)$,
            \item $\nu \in \dom(\rho_p), \eta \in \dom(\bar{\rho}_p)$ and $\rho_p(\eta) < \bar{\rho}_p(\eta)$.
        \end{itemize}
    \end{itemize}
    The ordering is given by $q \leq p$ if and only if $t_p \subseteq t_q$, $\rho_p \subseteq \rho_q$ and for each $\eta \in \dom(\bar{\rho}_p)$ we either have 
    \begin{itemize}
        \item  $\eta \in \dom(\bar{\rho}_q)$ and $\bar{\rho}_q(\eta) \geq \bar{\rho}_p(\eta)$ or
        \item $\eta \in \dom(\rho_q)$ and $\rho_q(\eta) \geq \bar{\rho}_p(\eta)$.
    \end{itemize}
\end{definition}

By definition, $\bS_\alpha'$ is a dense subforcing of $\bS_\alpha$, hence the difference between them is inconsequential regarding the generated forcing extension. The statement $\bar{\rho}_p(\eta) = \beta$ is interpreted as a promise that the (eventual) tag of $\eta$ in $t_G$ will be greater or equal to $\beta$.

\begin{lemma}
    $\bS_\alpha'$ is ${<}\kappa$-closed.
\end{lemma}

\begin{proof}
    Suppose $\langle t_i, \rho_i \rangle, i < \delta < \kappa$ is a decreasing sequence of conditions in $\bS_\alpha'$. Then forming $\langle t , \rho \rangle = \bigcup_{i < \delta} \langle t_i, \rho_i \rangle$, the only thing that needs to be checked is that $t$ is well-founded. But if $\seq{\eta_j}{j < \omega}$ were a branch of $t$, then $\seq{\rho(\eta_j)}{j < \omega}$ would form a strictly decreasing sequence of ordinals. 
\end{proof}

Hence $\bS_\alpha$ is strategically ${<}\kappa$-closed, cf.\ \cref{fact: strat closed}.

For a condition $p \in \bS_\alpha$ we define the rank
\[
    \crank(p) \defas \sup \setOf{\max(\rho_p(\eta), \bar{\rho}_p(\nu))}{\eta \in \dom(\rho_p), \nu \in \dom(\bar{\rho}_p)}.
\]
Furthermore, for $p \in \bS_\alpha$ and an ordinal $\beta$ define
\[
    \pi(p,\beta) \defas \langle t_p, \rho', \bar{\rho}' \rangle
\]
with $\rho' \defas \rho_p \restriction \setOf{\eta \in \dom(\rho_p)}{\rho_p(\eta) < \beta}$ and 
\[
    \bar{\rho}'(\eta) \defas
    \begin{cases}
        \bar{\rho}_p(\eta) & \text{ if } \eta \in \dom(\bar{\rho}_p) \text{ and } \bar{\rho}_p(\eta) < \beta \\
        \max \left(\beta, \sup \left\{
        \begin{gathered}
            \max(\rho'(\nu_1) + 1, \bar{\rho}'(\nu_2) + 1) : \\
            \nu_1 \in \succ(\eta) \cap \dom(\rho'),\\ 
            \nu_2 \in \succ(\eta) \cap \dom(\bar{\rho}')
        \end{gathered}
        \right\} \right) & \text{ otherwise.}
    \end{cases}
\]

for $\eta \in t_p \minus \dom(\rho')$. It is easy to see that for $\beta < \beta'$ we have 
\[
    p = \pi(p, \crank(p) + 1) \leq \pi(p,\beta') \leq \pi(p,\beta).
\]
Since $|t_p| < \kappa$, we can see by induction on the nodes of $t_p$ that $\crank(\pi(p,\beta)) < \beta + \kappa$.

\begin{theorem} \label{th: steel rank function}
    The function $\crank$ is a rank function on $\bS_\alpha$ with
    and $h_{\crank}(\beta) = \beta + \kappa$. In fact, if $\beta < \alpha$, for every $p \in \bS_\alpha$ and $\beta$ we may choose $q = \pi(p,\beta)$ in Definition~\ref{def: rank function}. 
\end{theorem}

\begin{proof}
    Fix a condition $p \in \bS_\alpha$ and an ordinal $\beta < \alpha$; let $q \defas \pi(p,\beta)$. Note that we have already remarked that $p \leq q$ and $\crank(q) < \beta + \kappa$. We wish to show that for every $r \in \bS_\alpha, \crank(r) < \beta$ we have 
    \[
    r \comp q \implies r \comp p.
    \]

    Let such an $r$ be given. To see that $r \comp p$, we will construct a $\hat{p} \leq r,p$ with $\hat{p} \in \bS_\alpha'$. Define $\hat{p} = \langle t_p \cup t_r, \rho', \emptyset \rangle$, where $\rho'(\eta)$ is simply defined as the maximum of $\rho_p(\eta), \bar{\rho}_p(\eta), \rho_r(\eta), \bar{\rho}_r(\eta)$, or only those of these values which are defined. By making, for each $\eta$, a case distinction on the value at which this maximum is attained, we can see that $\hat{p}$ is a condition. As an example, let us assume $\nu \in \succ(\eta)$ and $\rho'(\nu) = \rho_p(\nu)$. Then we know either $\rho_p(\eta) > \rho_p(\nu)$ or $\bar{\rho}_p(\eta) > \rho_p(\nu)$. In any case, $\rho'(\nu) < \max(\rho_p(\eta), \bar{\rho}_p(\eta)) \leq \rho'(\eta)$. Other cases follow similarly. 
    
    The nontrivial part is to show $\hat{p} \leq p,r$. The only way $\hat{p} \leq p$ could fail is if there is an $\eta \in t_p \cap \dom(\rho_p)$ such that $\rho'(\eta) > \rho_p(\eta)$, so let us assume this is the case. In particular, this implies $\eta \in t_r$. We distinguish cases.
    \begin{itemize}
        \item If $\rho_p(\eta) \geq \beta$, then we know $\eta \in \dom(\bar{\rho}_q)$ and $\bar{\rho}_q(\eta) \geq \beta$. But since $r \comp q$ and $\crank(r) < \beta$, this implies $\eta \in \dom(\bar{\rho}_r)$ and $\bar{\rho}_r(\eta) < \beta$. Overall $\rho'(\eta) = \rho_p(\eta)$.
        \item If $\rho_p(\eta) < \beta$, then $\rho_q(\eta) = \rho_p(\eta)$. This immediately implies that either $\eta \in \dom(\rho_r)$ and $\rho_r(\eta) = \rho_p(\eta)$ or $\eta \in \dom(\bar{\rho}_r)$ and $\bar{\rho}_r(\eta) \leq \rho_p(\eta)$. In any case, $\rho'(\eta) = \rho_p(\eta)$.
    \end{itemize}
    
    Similarly, we also prove $\hat{p} \leq r$. If this were not the case, then there is an $\eta \in t_p \cap t_r$ such that $\rho_p(\eta) > \rho_r(\eta)$ or $\bar{\rho}_p(\eta) > \rho_r(\eta)$.
    \begin{itemize}
        \item If $\rho_r(\eta) < \rho_p(\eta) < \beta$, then $\rho_q(\eta) = \rho_p(\eta)$ and thus $q \incomp r$.
        \item If $\rho_r(\eta) < \bar{\rho}_p(\eta) < \beta$, then again $\bar{\rho}_q(\eta) = \bar{\rho}_p(\eta)$, hence $q \incomp r$.
        \item If $\rho_p(\eta) \geq \beta$ or $\bar{\rho}_p(\eta) \geq \beta$, then $\bar{\rho}_q(\eta) \geq \beta$. But $\crank(r) < \beta$, hence $q \incomp r$. \qedhere
    \end{itemize}    
\end{proof}

For some $\beta$ we can refine the previous result if we only consider conditions $r$ of a special form.

\begin{lemma}
    Let $p \in \bS_\alpha$ and $\beta < \alpha$. Setting
    \[
    \beta' \defas \sup \setOf{\max(\rho_p(\eta), \bar{\rho}_p(\nu))}{\rho_p(\eta) < \beta + \kappa, \bar{\rho}_p(\nu) < \beta + \kappa} + 1,
    \]
    define $q \defas \pi(p, \beta')$. Then for every condition $r$ of the form $r = \pi(t, \beta)$ for some $t \in \bS_\alpha$ we have
    \[
        t \comp q \Rightarrow t \comp p.
    \]
\end{lemma}

\begin{proof}
    Similar to Theorem~\ref{th: steel rank function}. For a given condition $p$ and $r = \pi(t,\beta)$ with $r \comp q$, let us once again define $\hat{p}$ as before; we claim $\hat{p} \leq p,r$. If this is not true, there must exist an $\eta \in t_p \cap t_r$ such that one of the following is the case:
    \begin{itemize}
        \item $\rho_p(\eta) < \rho_r(\eta)$: Since $r = \pi(t,\beta)$, we have $\rho_r(\eta) < \beta$, this would imply $\rho_p(\eta) < \beta$ and thus $\rho_q(\eta) = \rho_p(\eta)$, hence $q \incomp r$.
        \item $\rho_p(\eta) < \bar{\rho}_r(\eta)$: In particular $\rho_p(\eta) < \beta + \kappa$. This means $\rho_p(\eta) < \beta'$ by definition of $\beta'$, and hence $\rho_q(\eta) = \rho_p(\eta)$.
        \item $\rho_r(\eta) < \rho_p(\eta) < \beta'$: In this case $\rho_q(\eta) = \rho_p(\eta)$, so $q \incomp r$.
        \item $\rho_r(\eta) < \rho_p(\eta)$ and $\rho_p(\eta) \geq \beta + \kappa$: We have $\bar{\rho}_q(\eta) \geq \beta'$ but $\rho_r(\eta) < \beta$, which cannot be.
        \item $\rho_r(\eta) < \bar{\rho}_p(\eta) < \beta'$: Once again $\bar{\rho}_q(\eta) = \bar{\rho}_p(\eta)$, which means $q \incomp r$.
        \item $\rho_r(\eta) < \bar{\rho}_p(\eta)$ and $\bar{\rho}_p(\eta) \geq \beta + \kappa$: We have $\bar{\rho}_q \geq \beta'$, but $\rho_r(\eta) < \beta$, contradiction. \qedhere
    \end{itemize}
\end{proof}

In the literature, the fact that $\crank$ from Theorem~\ref{th: steel rank function} is a rank function and thus Theorem~\ref{th: rank-forcing} takes hold for this language is usually presented in the form of the \textit{retagging lemma} \cite[Lemma~1]{steelForcingTaggedTrees1978}. A prominent use of the forcing was in Harrington's work on the connection between analytic determinacy and the existence of sharp for reals \cite{harrington_analytic_1978}. An analysis of the complexity of sentences in the language considered by Steel shows that the bounds derived here are more efficient that those in \cite[Lemma~1]{steelForcingTaggedTrees1978}. However, there is still room for improvement; the next subsection will contain even sharper bounds on the rank of the conditions necessary to force certain sentences, together with a proof of their optimality. 

\subsection{Calculating the complexity of well-founded trees} \label{subsec: wf trees}

Throughout this section we are dealing with $\kappa$-Borel subsets of $\powset(\pre{<\omega}{\kappa})$. Recall that $\powset(\pre{<\omega}{\kappa})$ carries a topology derived from the natural isomorphism to $\pre{\kappa}{2}$. Concretely, this topology is generated by a family of basic clopen sets which consists of all ${<}\kappa$-sized intersections of sets of the form $\setOf{T \subseteq \pre{<\omega}{\kappa}}{\eta \in T}$ and $\setOf{T \subseteq \pre{<\omega}{\kappa}}{\eta \notin T}$ for $\eta \in \pre{<\omega}{\kappa}$.

\begin{definition}
    For an ordinal $\alpha \leq \kappa^+$ we set
    \[
        WF_{\alpha} = \{x \subseteq \pre{<\omega}{\kappa} : x \text{ is a well-founded tree of rank ${<}\alpha$}\}.
    \]
    
\end{definition}

By the same argument as in \cite[Lemma~25.10]{jechSetTheory2003}, it is known that $WF_{<\alpha}$ is $\kappa$-Borel in $\powset(\pre{<\omega}{\kappa}) \simeq 2^\kappa$ for any $\alpha < \kappa^+$. The exact Borel complexity of these sets has in the classical case been calculated by Stern \cite{stern} and is intimately connected with rank arguments over Steel's forcing. In this section, we will show that analogous bound hold in the generalized setting. Here it should be noted that investigating well-founded trees is a matter where the generalized setting would ordinarily diverge from the classical one. The class of all well-founded subtrees of $\pre{<\omega}{\omega}$ is a complete\footnote{A set $X$ is complete for a class $\boldsymbol{\Gamma}$ if $X \in \boldsymbol{\Gamma}$ and every set $Y \in \boldsymbol{\Gamma}$ is a continuous preimage of $X$.} set for the class $\ppi{1}{1}{\powset(\pre{<\omega}{\omega})}{\omega}$, hence it in particular is not Borel. On the contrary, the class $WF_{\kappa^+}$ of all well-founded subtrees of $\pre{<\omega}{\kappa}$ is \textit{closed} in $\powset(\pre{<\omega}{\kappa})$, since having an infinite branch is coded in an initial ${<}\kappa$-sized segment of the tree. For this reason, it might not be expected that the classes of bounded rank still behave similarly.

For the duration of this section, let $\bar{A} = \setOf{\enquote{t_G \in B}}{B \text{ is a basic clopen set in } \powset(\pre{<\omega}{\kappa})}$.

\begin{lemma} \label{lem: steel set to form}
    Let $B$ be a (code for a) $\ppi{0}{\beta}{}{\kappa}$ subset of $\powset(\pre{<\omega}{\kappa})$. Then for $\alpha > \beta$ there exists a sentence $[\phi] \in \Pi_\beta(\mathcal{L}_\infty(\bar{A})/_{\bS_\alpha})$ such that
    \[
    \bS_\alpha \forces \phi \Leftrightarrow t_G \in B.
    \]
\end{lemma}

\begin{proof}
    The construction of the code $B$ immediately yields the sentence $\phi$. If $\beta = 0$, then we slightly abuse notation to let $\ppi{0}{0}{\powset(\pre{<\omega}{\kappa})}{\kappa}$ denote the class of basic clopen subsets of $\powset(\pre{<\omega}{\kappa})$; then there exists such a sentence $\phi$ by definition. Otherwise, $\beta > 0$ and $B = \bigcap_{i < \kappa} B_i$, where $B_i \in \bigcup_{\beta' < \beta} \ssigma{0}{\beta'}{\powset(\pre{<\omega}{\kappa})}{\kappa}$. By induction there exist sentences $[\phi_i] \in \bigcup_{\beta' < \beta} \Sigma_{\beta'}(\mathcal{L}_\infty(\bar{A})/_{\bS_\alpha})$ such that 
    \[
        \bS_\alpha \forces \phi_i \leftrightarrow t_G \in B_i
    \]
    for each $i < \kappa$. Hence $\phi \defas \bigwedge_{i < \kappa} \phi_i$ works. 
\end{proof}

\begin{lemma} \label{lem: steel forcing atoms}
    Let $p \in \bS_\alpha$ and $a \in \bar{A}$ such that $p \forces a$. Then, letting
    \[
        \beta = \sup \setOf{\max(\rho_p(\eta_1), \bar{\rho}_p(\eta_2))}{\rho_p(\eta_1) < \omega, \bar{\rho}_p(\eta_2) < \omega} + 1,
    \]
    we have $\pi(p,\beta) \forces a$.
\end{lemma}

\begin{proof}
    It is enough to consider $a = \enquote{\eta \in t_G}$ and $a = \enquote{\eta \notin t_G}$ for some $\eta \in \pre{<\omega}{\kappa}$.

    Note that $p \forces \eta \in t_G$ if and only if $\pred(\eta) \in t_p$ and $\rho(\pred(\eta)) > 0 \vee \bar{\rho}(\pred(\eta)) > 0$. In either case, $p \forces \eta \in t_G$ implies $\pi(p,1) \forces \eta \in t_G$. On the other hand, $p \forces \eta \notin t_G$ if and only if there exists an $i < n = \dom(\eta)$ such that $\eta \restriction i \in t_p$ and $\rho_p(\eta \restriction i) < n - i$. Hence if $p \forces \eta \notin t_G$, then $\pi(p, \beta) \forces \eta \notin t_G$.
\end{proof}

Notably, the choice $\pi(p, \beta)$ of condition in the previous lemma does not depend on $a$. We exploit this additional structure to prove Lemma~\ref{lem: improved ranked forcing steel}, improving on the bounds yielded by Theorem~\ref{th: rank-forcing}. For $\kappa = \omega$, \cite[Lemma~1.3]{stern} asserts an analogous statement; unfortunately, we have been unable to locate an extant proof of this main auxiliary result of Stern's article. It appears to have been omitted from the published version. \cite{zafrany_borel_1989} and \cite{miller_borel_1983} credit an unpublished work of Garland \cite{garland} for proving a stronger form of Theorem~\ref{th: complexity of wf trees} avoiding the use of forcing and instead constructing continuous reductions. Lastly, Zafrany \cite{zafrany_borel_1989} attributes the original proof of \cref{th: complexity of wf trees} for $\kappa=\omega$ to Luzin, whilst developing a framework to treat the complexity of an array of classes related to trees. We do not know whether Zafrany's methods generalize into the uncountable.

\begin{lemma} \label{lem: improved ranked forcing steel}
    Let $\alpha$ be large enough, $\beta \geq 1$ and $[\phi] \in \Pi_{2 \cdot \beta}(\mathcal{L}_\infty(\bar{A})/_{\bS_\alpha})$. If $p \forces \phi$, then already $\pi(p, \kappa \cdot \beta) \forces \phi$.
\end{lemma}

\begin{proof}
    We proceed by induction on $\beta$. The base case can be easily derived from the argument in the successor case $\beta \to \beta + 1$ for $\beta = 0$ together with the observation that whether $\kappa = \omega$ or $\kappa > \omega$, the condition yielded by Lemma~\ref{lem: steel forcing atoms} has rank ${<}\kappa$.
    
    Let us thus prove the successor step $\beta \to \beta + 1$ of the induction. Hence we have a $p,\phi$ with $[\phi] \in \Pi_{2 \cdot (\beta + 1)}(\mathcal{L}_\infty(\bar{A})/_{\bS_\alpha})$ and $p \forces \phi$. Now $\bS_\alpha \forces \phi \Leftrightarrow \bigwedge_{i \in I} \bigvee_{j \in J_i} \phi^i_j$, where $[\phi^i_j] \in \Pi_{2 \cdot \beta}(\mathcal{L}_\infty(\bar{A})/_{\bS_\alpha})$. If we can prove $\pi(p, \kappa \cdot (\beta +1)) \forces \bigvee_{j \in J_i} \phi^i_j$ for each $i \in I$, then we are done. Hence we fix $i \in I$ and write $\psi = \bigvee_{j \in J_i} \phi^i_j$ from now on. If $\pi(\beta, \kappa \cdot \beta + \kappa) \not \forces \psi$, then we can find an $s \leq \pi(\beta, \kappa \cdot \beta + \kappa)$ with $s \forces \neg \psi$. Set 
    \[
        s' \defas \pi(s, \sup \setOf{\max(\rho_s(\eta), \bar{\rho}_s(\nu))}{\rho_s(\eta) < \kappa \cdot \beta + \kappa, \bar{\rho}_s(\nu) < \kappa \cdot \beta + \kappa} + 1).
    \]
    We know $\crank(s') < \kappa \cdot \beta + \kappa$ and $s' \comp \pi(p, \kappa \cdot \beta + \kappa)$, thus $s' \comp p$ by Theorem~\ref{th: steel rank function}. Let $t \leq s', p$ be such that there exists a $j \in J_i$ with $t \forces \phi^i_j$. Since $[\phi^i_j] \in \Pi_{2 \cdot \beta}(\mathcal{L}_\infty(\bar{A})/_{\bS_\alpha})$, the inductive hypothesis gives us $\pi(t, \kappa \cdot \beta) \forces \phi^i_j$, But $t$ witnesses $s' \comp \pi(t, \kappa \cdot \beta)$, so Lemma~\ref{lem: improved ranked forcing steel} implies $s \comp \pi(t, \kappa \cdot \beta)$, which is absurd. In the base case $\beta = 0$, use Lemma~\ref{lem: steel forcing atoms} instead of the inductive hypothesis here.

    Now let $\gamma$ be a limit and $[\phi] \in \Pi_{2 \cdot \gamma}(\mathcal{L}_\infty(\bar{A})/_{\bS_\alpha}) = \Pi_{\gamma}(\mathcal{L}_\infty(\bar{A})/_{\bS_\alpha})$. Then $\bS_\alpha \forces \phi \Leftrightarrow \bigwedge_{i \in I} \phi_i$ with $[\phi_i] \in \bigcup_{\beta < \gamma} \Pi_{2 \cdot \beta}(\mathcal{L}_\infty(\bar{A})/_{\bS_\alpha})$. Furthermore assume $p \forces \phi$. But since $\pi(p, \kappa \cdot \gamma) \leq \pi(p, \kappa \cdot \beta)$ for all $\beta < \gamma$, the desired result is immediate.
\end{proof}

\begin{theorem} \label{th: complexity of wf trees}
    The sets $WF_\alpha$ have the following complexity as $\kappa$-Borel subsets of $\powset(\pre{<\omega}{\kappa})$:
    \begin{itemize}
        \item For every $\alpha < \kappa^+$ the set $WF_{\kappa \cdot \alpha}$ is $\ssigma{0}{2\cdot \alpha}{}{\kappa}$ but not $\ppi{0}{2 \cdot \alpha}{}{\kappa}$.
        \item For every $\alpha < \kappa^+$ and $0 < \beta < \kappa$ the set $WF_{\kappa \cdot \alpha + \beta}$ is $\ppi{0}{2\cdot \alpha + 1}{}{\kappa}$ but not $\ssigma{0}{2 \cdot \alpha + 1}{}{\kappa}$.
    \end{itemize}
\end{theorem}

\begin{proof}
    First, we introduce some notation and write $WF_\alpha(\eta)$ for the set of all $T \subseteq \pre{<\omega}{\kappa}$ such that $\eta \in T$ and
    \[
        \setOf{\nu \in \pre{<\omega}{\kappa}}{\eta \conc \nu \in T} \in WF_\alpha.
    \]
    We have $WF_\alpha = WF_\alpha(\emptyset)$.
    
    Then for any ordinal $\alpha < \kappa^+$ and $\eta \in \pre{<\omega}{\kappa}$ the equivalence
    \begin{equation} \label{eq: steel WF complexity}
        T \in WF_\alpha(\eta) \Leftrightarrow \eta \in T \wedge \exists \beta < \alpha\ \forall i < \kappa: \eta \conc i \notin T \vee T \in WF_{\beta}(\eta \conc i)
    \end{equation}
    holds. By noting $\cf(\kappa \cdot \alpha + \beta) < \kappa$ for $\beta < \kappa$ and using the fact that the class $\ppi{0}{2\cdot \alpha + 1}{}{\kappa}$ is closed under unions of size ${<}\kappa$, the positive assertions of the theorem now follow with a straightforward inductive argument from \eqref{eq: steel WF complexity}, proving $WF_{\kappa \cdot \alpha}(\eta) \in \ssigma{0}{2\cdot \alpha}{}{\kappa}$ and $WF_{\kappa \cdot \alpha + \beta} \in \ppi{0}{2\cdot \alpha + 1}{}{\kappa}$ for all $\alpha < \kappa^*, 0 < \beta < \kappa$ and $\eta \in \pre{<\omega}{\kappa}$.

    Now let us prove the negative assertions. First, we shall prove $WF_{\kappa \cdot \alpha + \beta} \notin \ssigma{0}{2 \cdot \alpha + 1}{}{\kappa}$. Assume this is false and let $B$ be a $\ssigma{0}{2 \cdot \alpha + 1}{}{\kappa}$ code for it; hence there exist $\ppi{0}{2 \cdot \alpha}{}{\kappa}$ codes $\seq{B_i}{i < \kappa}$ such that
    \[
        B = \bigcup_{i < \kappa} B_i,
    \]
    holds for the interpretations of these codes as $\kappa$-Borel sets. By $\ppi{1}{1}{}{\kappa}$-absoluteness for ${<}\kappa$-closed forcing notions we have $\bS_{\kappa \cdot \alpha + \beta + 1} \forces WF_{\kappa \cdot \alpha + \beta} = B$. Consider now the condition $p \in \bS_{\kappa \cdot \alpha + \beta + 1}$ with $t_p = {\emptyset}$ and $\rho_p(\emptyset) = \kappa \cdot \alpha$. Since $p \forces t_G \in WF_{\kappa \cdot \alpha + \beta} = B$, there exists a $p' \leq p$ with $p' \forces t_G \in B_{i^*}$ for an $i^* < \kappa$. Let $\phi$ be the $\Pi_{2 \cdot \alpha}(\mathcal{L}_\infty(\bar{A})/_{\bS_{\kappa \cdot \alpha + \beta + 1}})$ formula from Lemma~\ref{lem: steel set to form} for $B_{i^*}$. Lemma~\ref{lem: improved ranked forcing steel} yields $q \defas \pi(p', \kappa \cdot \alpha) \forces \phi$. But we know $\emptyset \in \dom(\bar{\rho}_{q})$ and $\bar{\rho}_q(\emptyset) < \kappa \cdot \alpha$, hence the condition $q' = \langle t_q, \rho_{q'}, \bar{\rho}_{q'} \rangle$ with $\rho_{q'}(\emptyset) = \kappa \cdot \alpha + \beta$ and 
    \[
        \rho_{q'} \restriction (t_q \minus \{\emptyset\}) = \rho_{q} \restriction (t_q \minus \{\emptyset\}), \bar{\rho}_{q'} \restriction (t_q \minus \{\emptyset\}) = \bar{\rho}_{q} \restriction (t_q \minus \{\emptyset\})
    \]
    satisfies $q' \leq q$. But this cannot be, since on the one hand we immediately get $q' \forces t_G \notin WF_{\kappa \cdot \alpha + \beta}$ and on the other hand we have $q' \forces \phi$, hence $q' \forces t_G \in B_{i^*}$ while
    \[
       \bS_{\kappa \cdot \alpha + \beta + 1} \forces WF_{\kappa \cdot \alpha + \beta} = B = \bigcup_{i < \kappa} B_i.
    \]

    We move now to the other part of the negative assertion, namely $WF_{\kappa \cdot \alpha} \notin \ppi{0}{2 \cdot \alpha}{}{\kappa}$. If this were the case for some $\alpha$, then \eqref{eq: steel WF complexity} gives us $WF_{\kappa \cdot \alpha + 1} \in \ppi{0}{2 \cdot \alpha}{}{\kappa} \subseteq \ssigma{0}{2 \cdot \alpha + 1}{}{\kappa}$, which is false.
\end{proof}

\newpage
\printbibliography
\end{document}